\documentclass[11pt]{article}
\usepackage{amsmath,amssymb,amsthm,amscd}

\newtheorem{thm}{Theorem}[section]
\newtheorem{lem}[thm]{Lemma}
\newtheorem{prop}[thm]{Proposition}
\newtheorem{cor}[thm]{Corollary}

\theoremstyle{definition}
\newtheorem{df}[thm]{Definition}
\newtheorem{rem}[thm]{Remark}

\numberwithin{equation}{section}

\renewcommand{\phi}{\varphi}

\newcommand{\ep}{\varepsilon}

\newcommand{\Ad}{\operatorname{Ad}}
\newcommand{\Aff}{\operatorname{Aff}}
\newcommand{\Bott}{\operatorname{Bott}}
\newcommand{\Ext}{\operatorname{Ext}}
\newcommand{\diag}{\operatorname{diag}}
\newcommand{\Hom}{\operatorname{Hom}}
\newcommand{\id}{\operatorname{id}}
\newcommand{\Ima}{\operatorname{Im}}
\newcommand{\Ker}{\operatorname{Ker}}
\newcommand{\Lip}{\operatorname{Lip}}
\newcommand{\Pext}{\operatorname{Pext}}
\newcommand{\Tor}{\operatorname{Tor}}
\newcommand{\Tr}{\operatorname{Tr}}
\newcommand{\tr}{\operatorname{tr}}

\newcommand{\N}{\mathbb{N}}
\newcommand{\Z}{\mathbb{Z}}
\newcommand{\Q}{\mathbb{Q}}
\newcommand{\R}{\mathbb{R}}
\newcommand{\C}{\mathbb{C}}
\newcommand{\T}{\mathbb{T}}
\newcommand{\K}{\mathbb{K}}

\title{Classification of homomorphisms \\
into simple $\mathcal{Z}$-stable $C^*$-algebras}
\author{Hiroki Matui \\
Graduate School of Science \\
Chiba University \\
Inage-ku, Chiba 263-8522, Japan}
\date{}

\begin{document}
\maketitle

\begin{abstract}
We classify unital monomorphisms 
into certain simple $\mathcal{Z}$-stable $C^*$-algebras 
up to approximate unitary equivalence. 
The domain algebra $C$ is allowed to be 
any unital separable commutative $C^*$-algebra, 
or any unital simple separable nuclear $\mathcal{Z}$-stable $C^*$-algebra 
satisfying the UCT such that $C\otimes B$ is of tracial rank zero 
for a UHF algebra $B$. 
The target algebra $A$ is allowed to be 
any unital simple separable $\mathcal{Z}$-stable $C^*$-algebra 
such that $A\otimes B$ has tracial rank zero 
for a UHF algebra $B$, 
or any unital simple separable exact $\mathcal{Z}$-stable $C^*$-algebra 
whose projections separate traces and 
whose extremal traces are finitely many. 
\end{abstract}

%%%%%%%%%%%%%%%%%%%%%%%%%%%%%%%%%%%%%%%%%%%%%%%%%%%%%%%%%%%%
\section{Introduction}

Consider unital monomorphisms $\phi,\psi:C\to A$ 
from a $C^*$-algebra $C$ to a simple $C^*$-algebra $A$. 
In this paper we study the problem to determine 
when $\phi$ and $\psi$ are approximately unitarily equivalent, 
i.e. when there exists a sequence of unitaries $(u_n)_n$ in $A$ 
such that $\phi(x)=\lim u_n\psi(x)u_n^*$ holds for any $x\in C$. 
This problem is known to be closely related to the classification problem 
for the simple $C^*$-algebra $A$. 
In the recent progress of Elliott's program to classify nuclear 
$C^*$-algebras via $K$-theoretic invariants 
(see \cite{Rordamtext} for an introduction to this subject), 
the Jiang-Su algebra plays a central role. 
The Jiang-Su algebra $\mathcal{Z}$, 
which was introduced by X. Jiang and H. Su in \cite{JS}, 
is a unital, simple, separable, infinite dimensional, 
stably finite and nuclear $C^*$-algebra $KK$-equivalent to $\C$. 
A $C^*$-algebra $A$ is said to be $\mathcal{Z}$-stable 
if $A\otimes\mathcal{Z}$ is isomorphic to $A$. 
$\mathcal{Z}$-stability implies many nice properties 
from the point of view of classification theory. 
Among other things, 
if $A$ is a unital separable simple $\mathcal{Z}$-stable $C^*$-algebra, 
then $A$ is either purely infinite or stably finite. 
If, in addition, $A$ is stably finite, 
then $A$ must have stable rank one and 
weakly unperforated $K_0(A)$ (see \cite{JS,GJS,R04IJM}). 
All classes of unital simple infinite dimensional $C^*$-algebras 
for which Elliott's classification conjecture is confirmed 
consist of $\mathcal{Z}$-stable algebras. 
It is then natural to consider classification of unital monomorphisms 
from certain $C^*$-algebras into simple $\mathcal{Z}$-stable $C^*$-algebras 
which are not necessarily of real rank zero. 
In the present paper we give a positive solution 
for large classes of unital stably finite $C^*$-algebras 
(Theorem \ref{main}, Corollary \ref{mainAH}, Theorem \ref{main2}). 

Classification of homomorphisms from $C(X)$ into a unital simple algebra 
has a long history. 
The earliest result for this subject is 
the classical Brown-Douglas-Fillmore theory \cite{BDF}. 
They showed that two unital monomorphisms $\phi$ and $\psi$ 
from $C(X)$ to the Calkin algebra $B(H)/K(H)$ 
are unitarily equivalent if and only if $KK(\phi)=KK(\psi)$. 
M. Dadarlat \cite{D} showed that 
two monomorphisms from $C(X)$ to a unital simple purely infinite $C^*$-algebra 
are approximately unitarily equivalent if and only if 
they give the same element in $KL(C(X),A)$. 
In the case that the target algebra $A$ is stably finite, 
G. Gong and H. Lin \cite{GL00ActaSin} showed that 
for a unital simple separable $C^*$-algebra $A$ 
with real rank zero, stable rank one, weakly unperforated $K_0(A)$ 
and a unique quasitrace $\tau$, 
two unital monomorphisms $\phi,\psi:C(X)\to A$ are 
approximately unitarily equivalent if and only if 
$KL(\phi)=KL(\psi)$ and $\tau\circ\phi=\tau\circ\psi$. 
H. Lin \cite{L07Trans} obtained the same result 
for the case that the target algebra $A$ is of tracial rank zero. 
P. W. Ng and W. Winter \cite{NW08} also obtained the same result 
for the case that $X$ is a path connected space and 
$A$ is a $\mathcal{Z}$-stable $C^*$-algebra of real rank zero. 
Similar classification up to approximate unitary equivalence 
is also known for more general domain algebras. 
G. A. Elliott \cite{E} showed that 
two homomorphisms $\phi$ and $\psi$ between AT algebras of real rank zero 
are approximately unitarily equivalent if and only if 
$K_i(\phi)=K_i(\psi)$ for each $i=0,1$. 
K. E. Nielsen and K. Thomsen \cite{NT} obtained the analogous result 
for general AT algebras. 
H. Lin \cite{L07Trans,L0801} classified 
unital homomorphisms from AH algebras 
into simple separable $C^*$-algebras of tracial rank no more than one. 
Classification up to asymptotic unitary equivalence is also studied 
in \cite{P,KK,L09AJM}. 

It should be noted that all the target algebras in these results 
are assumed to have many non-trivial projections 
(and most of them are of real rank zero). 
Indeed almost nothing is known so far 
when the target algebra does not contain non-trivial projections. 
The present paper gives a first non-trivial general result for this subject. 
Our target algebras consist of two classes $\mathcal{C}$ and $\mathcal{C}'$. 
The class $\mathcal{C}$ is the family of 
all unital simple separable $\mathcal{Z}$-stable $C^*$-algebras $A$ 
such that $A\otimes Q$ has tracial rank zero, 
where $Q$ denotes the universal UHF algebra. 
The classification theorems in \cite{W,LN} assert that 
any nuclear $C^*$-algebras $A,B\in\mathcal{C}$ satisfying the UCT 
are isomorphic if and only if 
their $K$-groups are isomorphic as graded ordered groups. 
The other class $\mathcal{C}'$ is the family of 
all unital simple separable stably finite $\mathcal{Z}$-stable 
exact $C^*$-algebras whose extremal traces are finitely many and 
whose projections separate traces. 
The Jiang-Su algebra $\mathcal{Z}$ itself is 
in $\mathcal{C}\cap\mathcal{C}'$ and 
any unital simple separable $\mathcal{Z}$-stable exact $C^*$-algebra 
with a unique trace is in $\mathcal{C}'$. 
In order to extend the target to 
$C^*$-algebras not necessarily of real rank zero, 
we need a new invariant $\Theta_{\phi,\psi}$, 
which is a homomorphism from $K_1(C)$ to $\Aff(T(A))/\Ima D_A$ 
(Lemma \ref{Theta}). 
Roughly speaking, if $A$ is of real rank zero, 
then the range of the dimension map $D_A$ is uniformly dense in $\Aff(T(A))$. 
Therefore this invariant trivially vanishes. 
When $A$ is not of real rank zero, 
it is not the case that $\Ima D_A$ is dense in $\Aff(T(A))$, 
and so the homomorphism $\Theta_{\phi,\psi}$ must be taken into account. 

The paper is organized as follows. 
In Section 2 we collect preliminary material. 
The most important one is the notion of $\Bott(\phi,u)$. 
In Section 3 we introduce the homomorphism $\Theta_{\phi,\psi}$ 
for a pair of unital monomorphisms $\phi,\psi:C\to A$. 
In Section 4 we give a classification theorem of unital monomorphisms 
from commutative $C^*$-algebras to 
certain unital simple $C^*$-algebras of real rank zero. 
The results in Section 4 (especially Theorem \ref{appunique1}) 
partly overlap with those obtained in \cite{L07Trans}. 
But the proof given in \cite{L07Trans} is quite lengthy, and so 
we provide a simpler and self-contained proof for the reader's convenience. 
Section 5 is devoted to the proof of a version of 
the so called basic homotopy lemma (see \cite{L0612}). 
In Section 6 we prove the classification theorem 
for the case that the domain algebra is commutative (Theorem \ref{main}) 
by combining the results obtained in Section 4 and 5. 
We also extend the classification theorem to the case that 
the domain is a unital AH algebra (Corollary \ref{mainAH}). 
In Section 7 we prove the classification theorem for the case that 
the domain is a nuclear $C^*$-algebra in $\mathcal{C}$ satisfying the UCT 
(Theorem \ref{main2}). 

\bigskip

\noindent
\textit{Acknowledgement. }
I would like to thank Huaxin Lin for valuable comments. 
I also like to thank the referee for a number of helpful comments.

%%%%%%%%%%%%%%%%%%%%%%%%%%%%%%%%%%%%%%%%%%%%%%%%%%%%%%%%%%%%
\section{Preliminaries}

\subsection{Notations}
We let $\log$ be the standard branch 
defined on the complement of the negative real axis.
For a Lipschitz continuous function $f$, 
we denote its Lipschitz constant by $\Lip(f)$. 
We denote by $\K$ 
the $C^*$-algebra of all compact operators on $\ell^2(\Z)$. 
The normalized trace on $M_n$ is written by $\tr$ and 
the unnormalized trace on $M_n$ or $\K$ is written by $\Tr$. 
The finite cyclic group of order $n$ is written by $\Z_n=\Z/n\Z$. 

Let $A$ be a $C^*$-algebra. 
For $a,b\in A$, we mean by $[a,b]$ the commutator $ab-ba$. 
We write $a\approx_\ep b$ when $\lVert a-b\rVert<\ep$. 
The set of tracial states on $A$ is denoted by $T(A)$ and 
the collection of all continuous bounded affine maps from $T(A)$ to $\R$ 
is denoted by $\Aff(T(A))$. 
We regard $\Aff(T(A))$ as a real Banach space with the sup norm. 
The dimension map $D_A:K_0(A)\to\Aff(T(A))$ is 
defined by $D_A([p])(\tau)=\tau(p)$. 
For a unital positive linear map $\phi:A\to B$ between unital $C^*$-algebras, 
$T(\phi):T(B)\to T(A)$ denotes the affine continuous map induced by $\phi$. 
We say that a $C^*$-algebra $A$ has strict comparison of projections 
if for projections $p,q\in A\otimes\K$, 
$(\tau\otimes\Tr)(p)<(\tau\otimes\Tr)(q)$ for any $\tau\in T(A)$ implies that 
$p$ is Murray-von Neumann equivalent to a subprojection of $q$. 
When $\phi$ is a homomorphism between $C^*$-algebras, 
$K_0(\phi)$ and $K_1(\phi)$ mean the induced homomorphisms on $K$-groups. 

A unital completely positive linear map is called a ucp map for short. 
A ucp map $\phi:A\to B$ is said to be $(G,\delta)$-multiplicative 
if 
\[
\lVert\phi(ab)-\phi(a)\phi(b)\rVert<\delta
\]
holds for any $a,b\in G$, where $G$ is a subset of $A$. 
For two ucp maps $\phi,\psi:A\to B$, 
we write $\phi\sim_{G,\delta}\psi$, 
when there exists a unitary $u\in B$ such that 
\[
\lVert\phi(a)-u\psi(a)u^*\rVert<\delta
\]
holds for any $a\in G$. 

\subsection{The entire $K$-group}
We recall the mod $p$ $K$-theory introduced by C. Schochet \cite{Sch}. 
The $K_i$-group of a $C^*$-algebra $A$ 
with the coefficient module $\Z_n$ for $i=0,1$, $n\in\N$ is defined by 
\[
K_i(A;\Z_n)=K_i(A\otimes\mathcal{O}_{n+1}), 
\]
where $\mathcal{O}_{n+1}$ is the Cuntz algebra. 
For notational convenience, we set $K_i(A;\Z_0)=K_i(A)$. 
Although our definition looks different from the original one in \cite{Sch}, 
it gives an equivalent theory to the conventional one 
(\cite[Theorem 6.4]{Sch}). 
The entire $K$-group $\underline{K}(A)$ of $A$ is defined by 
\[
\underline{K}(A)=\bigoplus_{n=0}^\infty(K_0(A;\Z_n)\oplus K_1(A;\Z_n)). 
\]
For each $i=0,1$ and $n\in\N$, we have the K\"unneth exact sequence 
\[
\begin{CD}
0@>>> K_i(A)\otimes\Z_n@>>> K_i(A;\Z_n)@>>>\Tor(K_i(A),\Z_n)@>>>0. 
\end{CD}
\]
It is known that this exact sequence splits unnaturally. 
For $C^*$-algebras $A,B$, 
we denote by $\Hom_\Lambda(\underline{K}(A),\underline{K}(B))$ 
the set of all homomorphisms from $\underline{K}(A)$ to $\underline{K}(B)$ 
preserving the direct sum decomposition and commuting with 
natural coefficient transformations and the Bockstein operations 
(see \cite{DL,Lintext} for details). 
M. Dadarlat and T. A. Loring \cite{DL} proved 
the following universal multicoefficient theorem. 

\begin{thm}\label{multiUCT}
Let $A$ be a $C^*$-algebra satisfying the UCT and 
let $B$ be a $\sigma$-unital $C^*$-algebra. 
Then there exists a short exact sequence 
\[
0\rightarrow\bigoplus_{i=0,1}\Pext(K_i(A),K_{1-i}(B))\rightarrow KK(A,B)
\rightarrow\Hom_\Lambda(\underline{K}(A),\underline{K}(B))\rightarrow 0, 
\]
where $\Pext(K_i(A),K_{1-i}(B))$ is the subgroup of $\Ext(K_i(A),K_{1-i}(B))$ 
consisting of the pure extensions. 
The sequence is natural in each variable. 
\end{thm}

Let $A$ and $B$ be $C^*$-algebras. 
Suppose that $A$ satisfies the UCT and $B$ is $\sigma$-unital. 
In \cite[Section 2.4]{Rordamtext}, 
the $KL$-group $KL(A,B)$ is defined 
as the quotient of $KK(A,B)$ by the image of $\Pext(K_*(A),K_{1-*}(B))$. 
Thus, by the theorem above, $KL(A,B)$ is identified with 
$\Hom_\Lambda(\underline{K}(A),\underline{K}(B))$. 
Throughout this paper we keep this identification. 
For a homomorphism $\phi:A\to B$, 
we denote by $K_i(\phi;\Z_n)$ the homomorphism 
from $K_i(A;\Z_n)$ to $K_i(B;\Z_n)$ induced by $\phi$. 
We set 
\[
KL(\phi)=(K_i(\phi;\Z_n))_{i,n}
\in\Hom_\Lambda(\underline{K}(A),\underline{K}(B)). 
\]
If $\phi:A\to B$ and $\psi:A\to B$ are approximately unitarily equivalent, 
then $KL(\phi)=KL(\psi)$ holds (see \cite{Rordamtext}). 
For $\kappa\in KL(A,B)=\Hom_\Lambda(\underline{K}(A),\underline{K}(B))$ and 
$i=0,1$, we denote its $K_i$-component by $K_i(\kappa)\in\Hom(K_i(A),K_i(B))$. 

\subsection{Almost multiplicative ucp maps}
For a $C^*$-algebra $A$, we mean by $P(A)$ the set of all projections of $A$. 
When $A$ is unital, 
we mean by $U(A)$ the set of all unitaries of $A$. 
The connected component of the identity in $U(A)$ is denoted by $U(A)_0$. 
Let $U_\infty(A)$ be the union of $U(A\otimes M_n)$'s via the embedding 
\[
U(A\otimes M_n)\ni u
\mapsto\begin{pmatrix}u&0\\0&1\end{pmatrix}\in U(A\otimes M_{n+1}). 
\]
Likewise, 
we let $U_\infty(A)_0$ denote the union of $U(A\otimes M_n)_0$'s. 

For a unital $C^*$-algebra $A$, we set 
\[
\mathcal{K}_0(A)=P(A\otimes\K)
\cup\bigcup_{n=1}^\infty P(A\otimes\mathcal{O}_{n+1}), 
\]
\[
\mathcal{K}_1(A)=U_\infty(A)
\cup\bigcup_{n=1}^\infty U(A\otimes\mathcal{O}_{n+1})
\]
and $\mathcal{K}(A)=\mathcal{K}_0(A)\cup\mathcal{K}_1(A)$. 
Let $\phi:A\to B$ be a $(G,\delta)$-multiplicative ucp map. 
For $p\in\mathcal{K}_0(A)$, 
if $G$ is sufficiently large and $\delta$ is sufficiently small, 
then $(\phi\otimes\id)(p)$ is close to a projection and 
one can consider its equivalence class in $K_0(B;\Z_n)$. 
We denote this class by $\phi_\#(p)$. 
In a similar fashion, for $u\in\mathcal{K}_1(A)$, 
if $G$ is sufficiently large and $\delta$ is sufficiently small, 
then $(\phi\otimes\id)(u)$ is close to a unitary and 
one can consider its equivalence class in $K_1(B;\Z_n)$. 
We denote this class by $\phi_\#(u)$. 
Thus, for any finite subset $L\subset\mathcal{K}(A)$, 
if $\phi$ is a sufficiently multiplicative ucp map, then 
$\phi_\#|L:L\to\underline{K}(B)$ is well-defined. 
In this paper, whenever we write $\phi_\#(x)$ or $\phi_\#|L$, 
the ucp map $\phi$ is always assumed to be sufficiently multiplicative 
so that they are well-defined. 
When $\phi$ is sufficiently multiplicative, 
we can verify the following easily: 
$\phi_\#(p)=\phi_\#(q)$ 
for Murray-von Neumann equivalent projections $p,q\in\mathcal{K}_0(A)$, 
$\phi_\#(p+q)=\phi_\#(p)+\phi_\#(q)$ 
for orthogonal projections $p,q\in\mathcal{K}_0(A)$, 
$\phi_\#(u)=0$ 
for any $u\in U_\infty(A)_0\cup U(A\otimes\mathcal{O}_{n+1})_0$ 
and 
$\phi_\#(uv)=\phi_\#(u)+\phi_\#(v)$ 
for any $u,v\in\mathcal{K}_1(A)$. 
Therefore $\phi_\#$ gives rise to a `partial homomorphism' 
from $\underline{K}(A)$ to $\underline{K}(B)$. 

Next, we would like to recall the notion of $\Bott(\phi,w)$ 
introduced in \cite{L0612}. 
Let $\phi:A\to B$ be a unital homomorphism between unital $C^*$-algebras 
and let $w\in B$ be a unitary satisfying 
\[
\lVert[\phi(a),w]\rVert<\delta
\]
for every $a\in G$, 
where $G$ is a large finite subset of $A$ and 
$\delta>0$ is a small positive real number. 
For a projection $p\in A\otimes C$, 
$(w\otimes1)(\phi\otimes\id)(p)+(\phi\otimes\id)(1-p)$ in $B\otimes C$ 
is close to a unitary, where $C$ is $M_n$ or $\mathcal{O}_{n+1}$. 
We denote the equivalence class of this unitary 
by $\Bott(\phi,w)(p)\in K_1(A\otimes C)$. 
Next, we would like to introduce $\Bott(\phi,w)(u)\in K_0(A\otimes C)$ 
for a unitary $u\in A\otimes C$. 
To this end we need to recall 
the notion of Bott elements associated with almost commuting unitaries 
(\cite[2.11]{L0612}). 
There exists a universal constant $\delta_0>0$ such that 
for any unitaries $v_1,v_2$ in a $C^*$-algebra $D$ 
satisfying $\lVert[v_1,v_2]\rVert<\delta_0$, 
the self-adjoint element 
\[
e(v_1,v_2)=\begin{pmatrix}f(v_1)&g(v_1)+h(v_1)v_2\\
g(v_1)+v_2^*h(v_1)&1-f(v_1)\end{pmatrix}\in M_2(D)
\]
has a spectral gap at $1/2$, 
where $f,g,h$ are certain universal real-valued continuous functions on $\T$ 
(\cite[Section 3]{Loring}). 
Then one can consider the $K_0$-class 
\[
[1_{[1/2,\infty)}(e(v_1,v_2))]
-\left[\begin{pmatrix}1&0\\0&0\end{pmatrix}\right]\in K_0(D)
\]
and call it the Bott element associated with $v_1,v_2$. 
In our setting, for a unitary $u\in A\otimes C$, 
we can consider the Bott element in $K_0(A\otimes C)$ corresponding to 
the almost commuting unitaries $(\phi\otimes\id)(u)$ and $w\otimes1$. 
We denote it by $\Bott(\phi,w)(u)\in K_0(A\otimes C)$. 
Thus, for a finite subset $L\subset\mathcal{K}(A)$, 
when $G$ is large enough and $\delta$ is small enough, 
then $\Bott(\phi,w)|L:L\to\underline{K}(B)$ is well-defined. 
In this paper, whenever we write $\Bott(\phi,w)|L$, 
$G$ and $\delta$ are always assumed to be chosen 
so that $\Bott(\phi,w)|L$ is well-defined. 
In the same way as above, 
we can see that $\Bott(\phi,w)$ gives rise to 
a `partial homomorphism' from $K_i(A\otimes C)$ to $K_{1-i}(B\otimes C)$. 

\subsection{The target algebras}
We denote the Jiang-Su algebra by $\mathcal{Z}$ (\cite{JS}). 
When a $C^*$-algebra $A$ satisfies $A\cong A\otimes\mathcal{Z}$, 
we say that $A$ is $\mathcal{Z}$-stable. 
We let $Q$ denote the universal UHF algebra, that is, 
$Q$ is the UHF algebra satisfying $K_0(Q)=\Q$. 

We introduce four classes 
$\mathcal{T}$, $\mathcal{T}'$, $\mathcal{C}$ and $\mathcal{C}'$ 
of unital simple separable stably finite $C^*$-algebras as follows. 

\begin{df}
We define $\mathcal{T}$ to be the class of 
all infinite dimensional unital simple separable $C^*$-algebras 
with tracial rank zero. 
Let $\mathcal{T}'$ be the class of 
infinite dimensional unital simple separable exact $C^*$-algebras $A$ 
with real rank zero, stable rank one, weakly unperforated $K_0(A)$ and 
finitely many extremal tracial states. 
We let $\mathcal{C}$ be the class of 
unital simple separable $\mathcal{Z}$-stable $C^*$-algebras $A$ 
such that $A\otimes Q$ is in $\mathcal{T}$. 
Let $\mathcal{C}'$ be the class of 
unital simple separable stably finite $\mathcal{Z}$-stable exact 
$C^*$-algebras $A$ 
whose projections separate traces and whose extremal traces are finitely many. 
\end{df}

\begin{rem}\label{fourclasses}
\begin{enumerate}
\item Any $A\in\mathcal{T}$ has real rank zero, stable rank one, 
weakly unperforated $K_0(A)$ and strict comparison of projections 
(see \cite[Chapter 3]{Lintext}). 
\item Exactness of $A\in\mathcal{T}'$ is assumed only for the purpose 
of using the fact that any quasitrace on an exact $C^*$-algebra is a trace 
(\cite{H}). 
By \cite[Corollary 6.9.2]{Blacktext}, 
any $A\in\mathcal{T}'$ has strict comparison of projections. 
\item If $A\in\mathcal{C}$, then $A\otimes B$ has tracial rank zero 
for any UHF algebra $B$ by \cite[Lemma 2.4]{MS}, that is, 
$A\otimes B$ belongs to $\mathcal{T}$. 
\item Let $A\in\mathcal{C}'$ and let $B$ be a UHF algebra.  
Then $A\otimes B$ has real rank zero by \cite[Theorem 1.4 (f)]{BKR} and 
has stable rank one by \cite[Corollary 6.6]{R91JFA} (or \cite{R04IJM}). 
By \cite[Theorem 5.2]{R92JFA} (or \cite{GJS}), 
$K_0(A\otimes B)$ is weakly unperforated. 
It follows that $A\otimes B$ is in $\mathcal{T}'$. 
\item Of course, $\mathcal{Z}$ itself is in $\mathcal{C}\cap\mathcal{C}'$. 
\end{enumerate}
\end{rem}

To continue, we fix a notation. 
Let $A$ and $B$ be unital stably finite $C^*$-algebras and 
let $\xi\in\Hom(K_0(A),K_0(B))$. 
We say that $\xi$ is unital when $\xi([1])=[1]$. 
We say that $\xi$ is positive (resp. strictly positive) 
when $\xi(K_0(A)_+)\subset K_0(B)_+$ 
(resp. $\xi(K_0(A)_+\setminus\{0\})\subset K_0(B)_+\setminus\{0\}$). 
Assume further that $A$ satisfies the UCT. 
We denote by $KL(A,B)_{+,1}$ the set of all $\kappa\in KL(A,B)$ 
such that $K_0(\kappa)$ is unital and strictly positive. 

\begin{lem}\label{KL+1}
Let $X$ be a connected compact metrizable space and 
let $B$ be a unital stably finite $C^*$-algebra. 
For $\xi\in\Hom(K_0(C(X)),K_0(B))$ the following are equivalent. 
\begin{enumerate}
\item $\xi$ is unital and strictly positive. 
\item $\xi$ is unital and positive. 
\end{enumerate}
If $K_0(B)$ is simple and weakly unperforated, then 
the two conditions above are equivalent to the following condition. 
\begin{enumerate}\setcounter{enumi}{2}
\item $\xi$ is unital and $\xi(\Ker D_{C(X)})\subset\Ker D_B$. 
\end{enumerate}
\end{lem}
\begin{proof}
(1)$\Rightarrow$(2) is clear. 
To show (2)$\Rightarrow$(1), assume that $\xi$ is unital and positive. 
By \cite[Corollary 6.3.6]{Blacktext}, 
$K_0(C(X))$ is a simple ordered group. 
Hence for any $x\in K_0(C(X))_+\setminus\{0\}$ 
there exists $n\in\N$ such that $[1]\leq nx$. 
Then $[1]\leq n\xi(x)$ in $K_0(A)$, and so $\xi(x)\neq0$. 

Assume that $K_0(B)$ is simple and weakly unperforated. 
Let us show (2)$\Rightarrow$(3). 
Take $x\in\Ker D_{C(X)}$ and $\tau\in T(B)$. 
The composition of $\tau$ and $\xi$ is a state on $K_0(C(X))$. 
By \cite[Corollary 6.10.3 (e)]{Blacktext}, 
any state on $K_0(C(X))$ comes from a trace. 
Therefore $\tau(\xi(x))=0$. 
It remains to show that (3) implies (1). 
Take $x\in K_0(C(X))_+\setminus\{0\}$. 
Let $\rho$ be a state on $K_0(B)$. 
The composition of $\rho$ and $\xi$ is a state on $K_0(C(X))$. 
By \cite[Corollary 6.10.3 (e)]{Blacktext} 
it comes from a trace on $C(X)$, and so $\rho(\xi(x))>0$. 
By \cite[Theorem 6.8.5]{Blacktext}, 
$K_0(B)$ has the strict ordering from its states. 
It follows that $\xi(x)$ is in $K_0(B)\setminus\{0\}$. 
\end{proof}

We recall the following three theorems from \cite{L01K,NW08}. 

\begin{thm}[{\cite[Corollary 4.6]{L01K}}]\label{LinK}
Let $A$ be a unital simple separable $C^*$-algebra 
with real rank zero, stable rank one and weakly unperforated $K_0(A)$. 
Then there exist a unital simple separable AH algebra $B$ 
with real rank zero and slow dimension growth 
and a unital homomorphism $\phi:B\to A$ 
which induces a graded ordered isomorphism from $K_*(B)$ to $K_*(A)$. 
\end{thm}

\begin{thm}[{\cite[Theorem 0.1]{NW08}}]\label{NWexist}
Let $X$ be a path connected compact metrizable space and 
let $A$ be a unital simple separable exact $C^*$-algebra 
with real rank zero, stable rank one and weakly unperforated $K_0(A)$. 
Let $\kappa\in KL(C(X),A)_{+,1}$ and 
let $\lambda:T(A)\to T(C(X))$ be an affine continuous map such that 
$\lambda(\tau)$ gives a strictly positive measure on $X$ 
for any $\tau\in T(A)$. 
Then there exists a unital monomorphism $\phi:C(X)\to A$ such that 
$KL(\phi)=\kappa$ and $T(\phi)=\lambda$. 
\end{thm}

\begin{thm}[{\cite[Theorem 0.2]{NW08}}]\label{NWunique}
Let $X$ be a path connected compact metrizable space and 
let $A$ be a unital simple separable exact $\mathcal{Z}$-stable 
$C^*$-algebra with real rank zero. 
Let $\phi,\psi:C(X)\to A$ be unital monomorphisms. 
Then $\phi$ and $\psi$ are approximately unitarily equivalent 
if and only if $KL(\phi)=KL(\psi)$ and 
$\tau\circ\phi=\tau\circ\psi$ for all $\tau\in T(A)$. 
\end{thm}

\begin{rem}
\begin{enumerate}
\item The proof of \cite[Theorem 0.1]{NW08} uses \cite[Corollary 4.6]{L01K}, 
and in the statement of \cite[Corollary 4.6]{L01K} 
$A$ is assumed to be nuclear. 
But the proof given there does not use nuclearity, and so we omit it. 
In the statement of \cite[Theorem 0.2]{NW08}, 
$A$ is also assumed to be nuclear. 
But its proof needs only exactness of $A$. 
\item The condition (b) of \cite[Theorem 0.1]{NW08} automatically 
follows from other assumptions, 
because any traces on $C(X)$ induce the same state on $K_0(C(X))$ 
(\cite[Corollary 6.10.3 (a)]{Blacktext}) and 
$K_0(C(X))$ has no other states (\cite[Corollary 6.10.3 (e)]{Blacktext}). 
\end{enumerate}
\end{rem}

We give a generalization of Theorem \ref{NWexist} for later use. 

\begin{cor}\label{NWexist2}
Let $C=\bigoplus_{i=1}^np_i(C(X_i)\otimes M_{k_i})p_i$, 
where $X_i$ is a path connected compact metrizable space and 
$p_i\in C(X_i)\otimes M_{k_i}$ is a projection. 
Let $A$ be a unital simple separable exact $C^*$-algebra 
with real rank zero, stable rank one and weakly unperforated $K_0(A)$. 
Let $\kappa\in KL(C,A)_{+,1}$ and 
let $\lambda:T(A)\to T(C)$ be an affine continuous map such that 
$\lambda(\tau)$ is a faithful trace for any $\tau\in T(A)$. 
Suppose that $\lambda(\tau)(p_i)=\tau(K_0(\kappa)([p_i]))$ holds 
for any $\tau\in T(A)$ and $i=1,2,\dots,n$. 
Then there exists a unital monomorphism $\phi:C\to A$ such that 
$KL(\phi)=\kappa$ and $T(\phi)=\lambda$. 
\end{cor}
\begin{proof}
It is clear that 
the case $C=C(X)\otimes M_k$ follows immediately from Theorem \ref{NWexist}. 
Let us consider the case $C=p(C(X)\otimes M_k)p$, 
where $X$ is a path connected compact metrizable space. 
Let $m\in\N$ be the rank of $p$. 
There exist $l\in\N$ and 
a projection $q\in C\otimes M_l\subset C(X)\otimes M_{kl}$ such that 
$p\otimes e$ is a subprojection of $q$ and $q$ is 
Murray-von Neumann equivalent to $1_{C(X)}\otimes r$, 
where $e\in M_l$ is a minimal projection of $M_l$ and 
$r\in M_{kl}$ is a projection of rank $k$. 
We can find a projection $\tilde q\in A\otimes M_l$ 
such that $K_0(\kappa)([q])=[\tilde q]$. 
Set $C_0=q(C\otimes M_l)q\cong C(X)\otimes M_k$ and 
$A_0=\tilde q(A\otimes M_l)\tilde q$. 
For any tracial state $\tau\in T(A)$, 
$mk^{-1}(\tau\otimes\Tr)$ gives a tracial state on $A_0$, 
and this correspondence induces a homeomorphism 
between $T(A)$ and $T(A_0)$. 
Likewise there exists a natural homeomorphism 
between $T(C)$ and $T(C_0)$. 
The identifications 
\[
KL(C,A)_{+,1}\cong KL(C_0,A_0)_{+,1},\quad 
T(A)\cong T(A_0)\quad\text{and}\quad T(C)\cong T(C_0)
\]
allow us to regard $\kappa$ as an element of $KL(C_0,A_0)_{+,1}$ 
and $\lambda$ as an affine continuous map from $T(A_0)$ to $T(C_0)$. 
Therefore the previous case shows that 
there exists $\phi:C_0\to A_0$ realizing $\kappa$ and $\lambda$. 
From $[\phi(p\otimes e)]=K_0(\kappa)([p\otimes e])=[1_A\otimes e]$, 
there exists a unitary $u\in A\otimes M_l$ 
such that $u\phi(p\otimes e)u^*=1_A\otimes e$. 
The restriction of $\Ad u\circ\phi$ to $(p\otimes e)C_0(p\otimes e)$ 
gives a desired unital monomorphism from $C$ to $A$. 

We now turn to the general case. 
Let $C=\bigoplus_{i=1}^np_i(C(X_i)\otimes M_{k_i})p_i$, 
$\kappa$ and $\lambda$ be as in the statement. 
Let $\gamma_i:p_i(C(X_i)\otimes M_{k_i})p_i\to C$ be the canonical embedding. 
Choose projections $\tilde p_1,\tilde p_2,\dots,\tilde p_n\in A$ so that 
$K_0(\kappa)([p_i])=[\tilde p_i]$ and 
$\tilde p_1+\tilde p_2+\dots+\tilde p_n=1$. 
For each $i=1,2,\dots,n$, $\kappa\circ KL(\gamma_i)$ is regarded 
as an element of $KL(p_iC,\tilde p_iA\tilde p_i)_{+,1}$. 
For $\tau\in T(A)$, the restriction of $\tau/\tau(\tilde p_i)$ 
to $\tilde p_iA\tilde p_i$ is a tracial state. 
Similarly the restriction of $\lambda(\tau)/\tau(\tilde p_i)$ to $p_iC$ 
is also a tracial state, 
because $\lambda(\tau)(p_i)=\tau(K_0(\kappa)([p_i]))=\tau(\tilde p_i)$. 
It is not so hard to check that 
\[
\lambda_i:\tau/\tau(\tilde p_i)\mapsto\lambda(\tau)/\tau(\tilde p_i)
\]
gives rise to an affine continuous map 
from $T(\tilde p_iA\tilde p_i)$ to $T(p_iC)$. 
We have already shown that 
$\kappa\circ KL(\gamma_i)$ and $\lambda_i$ are realized 
by a unital monomorphism $\phi_i:p_iC\to\tilde p_iA\tilde p_i$. 
Then $\phi=\phi_1+\phi_2+\dots+\phi_n$ is a unital monomorphism 
from $C$ to $A$ satisfying $KL(\phi)=\kappa$ and $T(\phi)=\lambda$. 
\end{proof}

%%%%%%%%%%%%%%%%%%%%%%%%%%%%%%%%%%%%%%%%%%%%%%%%%%%%%%%%%%%%
\section{Determinants of unitaries}

In this section, we would like to introduce a homomorphism 
\[
\Theta_{\phi,\psi}:K_1(C)\to\Aff(T(A))/\Ima D_A, 
\]
which plays an important role in the main theorems of this paper 
(Theorem \ref{main}, Corollary \ref{mainAH}, Theorem \ref{main2}). 

Let $A$ be a unital $C^*$-algebra. 
For $\tau\in T(A)$, 
the de la Harpe-Skandalis determinant (\cite{HS}) associated with $\tau$ is 
written by 
\[
\Delta_\tau:U_\infty(A)_0\to\R/D_A(K_0(A))(\tau). 
\] 
It is well-known that 
$\Delta_A(u)(\tau)=\Delta_\tau(u)$ gives a homomorphism 
\[
\Delta_A:U_\infty(A)_0\to\Aff(T(A))/\Ima D_A. 
\]
Let $C,A$ be unital $C^*$-algebras and 
let $\phi,\psi:C\to A$ be unital homomorphisms 
satisfying $K_1(\phi)=K_1(\psi)$ and $T(\phi)=T(\psi)$. 
In what follows, we use the same notation $\phi,\psi$ for the homomorphisms 
from $C\otimes M_n$ to $A\otimes M_n$ induced by $\phi,\psi$. 
For $u\in U_\infty(C)$, 
we can consider $\Delta_A(\phi(u^*)\psi(u))$, 
as $\phi(u^*)\psi(u)$ belongs to $U_\infty(A)_0$. 

\begin{lem}\label{Theta}
In the setting above, we have the following. 
\begin{enumerate}
\item There exists a homomorphism 
\[
\Theta_{\phi,\psi}:K_1(C)\to\Aff(T(A))/\Ima D_A
\]
such that $\Theta_{\phi,\psi}([u])=\Delta_A(\phi(u^*)\psi(u))$ 
for any $u\in U_\infty(C)$. 
\item For any $w\in U(A)$, 
$\Theta_{\phi,\Ad w\circ\psi}=\Theta_{\phi,\psi}$. 
\item If $\phi$ and $\psi$ are approximately unitarily equivalent, 
then $\Ima\Theta_{\phi,\psi}\subset\overline{\Ima D_A}$. 
\item If $C$ satisfies the UCT and $KL(\phi)$ equals $KL(\psi)$, 
then the homomorphism $\Theta_{\phi,\psi}$ 
factors through $K_1(C)/\Tor(K_1(C))$. 
\end{enumerate}
\end{lem}
\begin{proof}
(1) We first show that 
$\Delta_A(\phi(u^*)\psi(u))$ equals $\Delta_A(\phi(v^*)\psi(v))$ 
when $u,v\in U_\infty(C)$ satisfy $uv^*\in U_\infty(C)_0$. 
We can find $n\in\N$ and 
piecewise smooth paths of unitaries 
$x:[0,1]\to U(A\otimes M_n)$, $y:[0,1]\to U(A\otimes M_n)$ and 
$z:[0,1]\to U(C\otimes M_n)$ such that 
$x(0)=\phi(u)$, $x(1)=\psi(u)$, 
$y(0)=\phi(v)$, $y(1)=\psi(v)$ and $z(0)=u$, $z(1)=v$. 
Define $h:[0,1]\to U(A\otimes M_n)$ by 
\[
h(t)=\begin{cases}x(4t)&1\leq t\leq1/4\\
\psi(z(4t-1))&1/4\leq t\leq1/2\\
y(3-4t)&1/2\leq t\leq3/4\\
\phi(z(4-4t))&3/4\leq t\leq1. \end{cases}
\]
Since $h$ is a closed path of unitaries, one has 
\[
\frac{1}{2\pi\sqrt{-1}}\int_0^1(\tau\otimes\Tr)(\dot{h}(t)h(t)^*)\,dt
\in D_A(K_0(A))(\tau)
\]
for any $\tau\in T(A)$. 
It is easy to see that 
the contribution from $t\mapsto\psi(z(4t-1))$ and $t\mapsto\phi(z(4-4t))$ 
cancels out, because of $T(\phi)=T(\psi)$. 
It follows that 
\[
\frac{1}{2\pi\sqrt{-1}}\left(\int_0^1(\tau\otimes\Tr)(\dot{x}(t)x(t)^*)\,dt
-\int_0^1(\tau\otimes\Tr)(\dot{y}(t)y(t)^*)\,dt\right)
\in D_A(K_0(A))(\tau)
\]
for any $\tau\in T(A)$, 
which implies $\Delta_A(\phi(u^*)\psi(u))=\Delta_A(\phi(v^*)\psi(v))$. 
It follows that 
$\Theta_{\phi,\psi}:K_1(C)\to\Aff(T(A))/\Ima D_A$ is well-defined as a map 
by $\Theta_{\phi,\psi}([u])=\Delta_A(\phi(u^*)\psi(u))$ 
for any $u\in U_\infty(C)$. 

For any $u,v\in U_\infty(C)$, 
$\diag(uv,1)$ is homotopic to $\diag(u,v)$, and so 
\begin{align*}
\Delta_A(\diag(\phi(uv)^*\psi(uv),1))
&=\Delta_A(\diag(\phi(u)^*\psi(u),\phi(v)^*\psi(v))) \\
&=\Delta_A(\phi(u^*)\psi(u))+\Delta_A(\phi(v^*)\psi(v)). 
\end{align*}
Hence we can conclude that $\Theta_{\phi,\psi}$ is a homomorphism. 

(2) can be shown in a similar fashion to the proof of 
(ii)$\Rightarrow$(i) of \cite[Theorem 3.1]{KK}. 
We leave the details to the reader. 

(3) follows from (1) and (2). 

(4) Let 
\[
M_{\phi,\psi}=\{f\in C([0,1],A)\mid f(0)=\phi(c),\ f(1)=\psi(c)\quad 
\text{for some }c\in C\}
\]
be the mapping torus of $\phi,\psi:C\to A$. 
Since $K_i(\phi)=K_i(\psi)$ for $i=0,1$, 
from the short exact sequence 
\[
\begin{CD}
0@>>>SA@>>>M_{\phi,\psi}@>\pi>>C@>>>0
\end{CD}
\]
of $C^*$-algebras, we obtain the following short exact sequence 
of abelian groups: 
\[
\begin{CD}
0@>>>K_0(A)@>>>K_1(M_{\phi,\psi})@>K_1(\pi)>>K_1(C)@>>>0. 
\end{CD}
\]
By $KL(\phi)=KL(\psi)$, this exact sequence is pure 
(see Theorem \ref{multiUCT}). 
Thus, the quotient map $K_1(\pi)$ has a right inverse 
on any finitely generated subgroup of $K_1(C)$. 
Let $R_{\phi,\psi}:K_1(M_{\phi,\psi})\to\Aff(T(A))$ be the rotation map 
introduced in \cite[Section 2]{L0612} (see also \cite[Section 1]{KK}). 
It is easy to verify that 
\[
R_{\phi,\psi}(x)+\Ima D_A=\Theta_{\phi,\psi}(K_1(\pi)(x)) 
\]
holds for every $x\in K_1(M_{\phi,\psi})$. 
Therefore $\Theta_{\phi,\psi}$ kills torsion of $K_1(A)$, 
because $\Aff(T(A))$ is torsion free. 
In other words, $\Theta_{\phi,\psi}$ factors through $K_1(C)/\Tor(K_1(C))$. 
\end{proof}

%%%%%%%%%%%%%%%%%%%%%%%%%%%%%%%%%%%%%%%%%%%%%%%%%%%%%%%%%%%%
\section{$C^*$-algebras of real rank zero}

In this section we give a classification result of 
unital monomorphisms from $C(X)$ to a $C^*$-algebra 
in $\mathcal{T}\cup\mathcal{T}'$ (Theorem \ref{rr0unique}). 
We begin with the following lemma, 
which is a variant of \cite[Lemma 2.2]{GL00ActaSin}. 
A similar argument is also found in \cite[Lemma 6.2.7]{Lintext}. 

\begin{lem}\label{matrix}
Let $X$ be a compact metrizable space. 
For any finite subset $F\subset C(X)$ and $\ep>0$, 
there exist a finite subset $G\subset C(X)$ and $\delta>0$ 
such that the following hold. 
Let $\phi$ and $\psi$ be $(G,\delta)$-multiplicative ucp maps 
from $C(X)$ to $M_n$ such that 
\[
\lvert\tr(\phi(f))-\tr(\psi(f))\rvert<\delta\quad \forall f\in G. 
\]
Then there exist a projection $p\in M_n$, 
$(F,\ep)$-multiplicative ucp maps $\phi',\psi':C(X)\to pM_np$ and 
a unital homomorphism $\sigma:C(X)\to(1{-}p)M_n(1{-}p)$ such that 
$\phi\sim_{F,\ep}\phi'\oplus\sigma$, $\psi\sim_{F,\ep}\psi'\oplus\sigma$ 
and $\tr(p)<\ep$. 
\end{lem}
\begin{proof}
Suppose that we are given a finite subset $F\subset C(X)$ and $\ep>0$. 
We may assume that elements of $F$ are of norm one. 
The proof is by contradiction. 
If the lemma was false, then 
we would have a sequence of pairs of ucp maps $\phi_n$ and $\psi_n$ 
from $C(X)$ to $M_{m_n}$ such that 
\[
\lVert\phi_n(fg)-\phi_n(f)\phi_n(g)\rVert\to0,\quad 
\lVert\psi_n(fg)-\psi_n(f)\psi_n(g)\rVert\to0
\]
and 
\[
\lvert\tr(\phi_n(f))-\tr(\psi_n(f))\rvert\to0
\]
as $n\to\infty$ for any $f,g\in C(X)$, 
and the conclusion of the lemma does not hold for any $\phi_n,\psi_n$. 
Let $\omega\in\beta\N\setminus\N$ be a free ultrafilter on $\N$. 
Define 
\[
\bigoplus_\omega M_{m_n}=\left\{(a_n)_n\in\prod M_{m_n}\mid
\lim_{n\to\omega}\lVert a_n\rVert=0\right\}. 
\]
We set $A=\prod M_{m_n}/\bigoplus_\omega M_{m_n}$ and 
let $\pi:\prod M_{m_n}\to A$ be the quotient map. 
Define ucp maps $\tilde\phi$ and $\tilde\psi$ from $C(X)$ to $\prod M_{m_n}$ 
by $\tilde\phi(f)=(\phi_n(f))_n$ and $\tilde\psi(f)=(\psi_n(f))_n$ 
for $f\in C(X)$. 
Clearly $\pi\circ\tilde\phi$ and $\pi\circ\tilde\psi$ are 
unital homomorphisms from $C(X)$ to $A$. 
One can define a tracial state $\tau\in T(A)$ by 
\[
\tau(\pi((a_n)_n))=\lim_{n\to\omega}\tr(a_n)
\]
for $(a_n)_n\in\prod M_{m_n}$. 
Then we have $\tau\circ\pi\circ\tilde\phi=\tau\circ\pi\circ\tilde\psi$. 
Let $\mu$ be the probability measure on $X$ 
corresponding to $\tau\circ\pi\circ\tilde\phi=\tau\circ\pi\circ\tilde\psi$. 

Any $x\in X$ has an open neighborhood $U_x$ such that 
$\mu(\overline{U_x}\setminus U_x)=0$ and 
$\lvert f(y)-f(y')\rvert<\ep/3$ for any $y,y'\in U_x$ and $f\in F$. 
(Such $U_x$ exists by the following reason. 
Let $d(\cdot,\cdot)$ be a metric compatible with the topology of $X$ and 
let $C_r=\{y\in X\mid d(x,y){=}r\}$ for $r>0$. 
There exist only countably many $r$ such that $\mu(C_r)>0$, 
because $\mu$ is a probability measure. 
Hence it is easy to find $r>0$ so that 
$\mu(C_r)=0$ and $\lvert f(x)-f(y)\rvert<\ep/6$ 
for any $y\in X$ with $d(x,y)<r$ and $f\in F$. 
Then $U_x=\{y\in X\mid d(x,y)<r\}$ meets the requirement.) 
Since $X$ is compact, we can find $x_1,x_2,\dots,x_k\in X$ such that 
$U_{x_1}\cup\dots\cup U_{x_k}=X$. 
Consider open subsets of the form $W=V_1\cap V_2\cap\dots\cap V_k$ 
satisfying $\mu(W)>0$, 
where $V_i$ is either $U_{x_i}$ or $X\setminus\overline{U_{x_i}}$. 
Let $W_1,W_2,\dots,W_l$ be these open subsets. 
Evidently $W_i$'s are pairwise disjoint. 
Then we have 
\[
\mu(X\setminus(W_1\cup W_2\cup\dots\cup W_l))=0
\]
and $\lvert f(y)-f(y')\rvert<\ep/3$ for any $y,y'\in W_i$ and $f\in F$. 
Choose $z_i\in W_i$ for each $i=1,2,\dots,l$. 
The $C^*$-algebra $\prod M_{m_n}$ has real rank zero and so does $A$. 
Accordingly, 
the hereditary subalgebra of $A$ generated by $\pi(\tilde\phi(C_0(W_i)))$ 
contains an approximate unit consisting of projections. 
It follows that there exists a projection 
\[
p_i\in\overline{\pi(\tilde\phi(C_0(W_i)))A\pi(\tilde\phi(C_0(W_i)))}
\]
satisfying $\tau(p_i)>\mu(W_i)-\ep/l$. 
It is easy to see that 
$\lVert\pi(\tilde\phi(f))p_i-f(z_i)p_i\rVert<\ep/3$ holds for any $f\in F$. 
Extend $\pi\circ\tilde\phi$ to a unital homomorphism 
from $C(X)^{**}$ to $A^{**}$ 
and define $\bar p_i=\pi(\tilde\phi(1_{W_i}))$. 
Then $\bar p_i$ commutes with $\pi(\tilde\phi(C(X)))$ and $p_i\leq\bar p_i$. 
Similarly one can find projections $q_i$ 
in the hereditary subalgebra generated by $\pi(\tilde\psi(C_0(W_i)))$ 
and $\bar q_i$ in $A^{**}\cap\pi(\tilde\psi(C(X)))'$ 
satisfying analogous properties. 

It is not so hard to find projections $p'_i\leq p_i$, $q'_i\leq q_i$ 
and a unitary $u\in A$ such that $p'_i=uq'_iu^*$ and 
\[
\tau(p'_i)=\tau(q'_i)=\min\{\tau(p_i),\tau(q_i)\}
\]
for any $i=1,2,\dots,l$. 
Set $p=1-(p'_1+p'_2+\dots+p'_l)$ and $q=1-(q'_1+q'_2+\dots+q'_l)$. 
We have 
\[
\tau(p)=1-\sum_{i=1}^l\tau(p'_i)<1-\sum_{i=1}^l(\mu(W_i)-\ep/l)=\ep. 
\]
Moreover, 
\begin{align*}
(1-p)\pi(\tilde\phi(f))
&=\sum_{i=1}^lp'_i\pi(\tilde\phi(f))
=\sum_{i=1}^lp'_i\bar p_i\pi(\tilde\phi(f))
=\sum_{i=1}^lp'_i\pi(\tilde\phi(f))\bar p_i \\
&\approx_{\ep/3}\sum_{i=1}^lf(z_i)p'_i\bar p_i
=\sum_{i=1}^lf(z_i)p'_i
\end{align*}
holds for any $f\in F$. 
Similarly one has 
$\lVert(1-q)\pi(\tilde\psi(f))-\sum_{i=1}^lf(z_i)q'_i\rVert<\ep/3$ 
for any $f\in F$. 
It is well-known that the projections $p'_i$, $p$ in $A$ lift to 
projections $(p'_{i,n})_n$, $(p_n)_n$ in $\prod M_{m_n}$ satisfying 
$p_n+p'_{1,n}+\dots+p'_{l,n}=1$. 
Similarly the unitary $u\in A$ lifts to a unitary $(u_n)_n\in\prod M_{m_n}$. 
Define ucp maps $\phi'_n,\psi'_n:C(X)\to p_nM_{m_n}p_n$ and 
a unital homomorphism $\sigma_n: C(X)\to(1{-}p_n)M_{m_n}(1{-}p_n)$ by 
\[
\phi'_n(f)=p_n\phi_n(f)p_n,\quad 
\psi'_n(f)=p_nu_n\psi_n(f)u_n^*p_n\quad\text{and}\quad 
\sigma_n(f)=\sum_{i=1}^lf(z_i)p'_{i,n}. 
\]
It follows that there exists $n\in\N$ such that 
$\phi'_n$ and $\psi'_n$ are $(F,2\ep/3)$-multiplicative, 
$\phi_n\sim_{F,\ep}\phi'_n\oplus\sigma_n$ and 
$\psi_n\sim_{F,\ep}\psi'_n\oplus\sigma_n$. 
This contradicts the assumption, and so the proof is completed. 
\end{proof}

We can prove the following lemma in the same way as above. 

\begin{lem}\label{rr0}
Let $X$ be a compact metrizable space. 
For any finite subset $F\subset C(X)$, $\ep>0$ and $m\in\N$, 
there exist a finite subset $G\subset C(X)$ and $\delta>0$ 
such that the following hold. 
Let $A\in\mathcal{T}'$ be a $C^*$-algebra 
with at most $m$ extremal tracial states. 
Let $\phi$ and $\psi$ be $(G,\delta)$-multiplicative ucp maps 
from $C(X)$ to $A$ such that 
\[
\lvert\tau(\phi(f))-\tau(\psi(f))\rvert<\delta\quad 
\forall f\in G,\ \tau\in T(A). 
\]
Then there exist a projection $p\in A$, 
$(F,\ep)$-multiplicative ucp maps $\phi',\psi':C(X)\to pAp$ and 
a unital homomorphism $\sigma:C(X)\to(1{-}p)A(1{-}p)$ 
with finite dimensional range such that 
$\phi\sim_{F,\ep}\phi'\oplus\sigma$, $\psi\sim_{F,\ep}\psi'\oplus\sigma$ 
and $\tau(p)<\ep$ for any $\tau\in T(A)$. 
\end{lem}
\begin{proof}
Suppose that 
we are given a finite subset $F\subset C(X)$, $\ep>0$ and $m\in\N$. 
We may assume that elements of $F$ are of norm one. 
The proof is by contradiction. 
If the lemma was false, then 
we would have a sequence of $C^*$-algebras $(A_n)_n$ in $\mathcal{T}'$ and 
a sequence of pairs of ucp maps $\phi_n$ and $\psi_n$ 
from $C(X)$ to $A_n$ as in the proof of Lemma \ref{matrix}. 
Define $\tilde\phi$, $\tilde\psi$, $B=\prod A_n/\bigoplus_\omega A_n$ and 
$\pi:\prod A_n\to B$ in the same way. 
For each $n$, choose extremal tracial states 
$\tau_{1,n},\tau_{2,n},\dots,\tau_{m,n}\in T(A_n)$ 
so that $\{\tau_{1,n},\tau_{2,n},\dots,\tau_{m,n}\}$ 
exhausts all the extremal traces on $A_n$. 
For each $j=1,2,\dots,m$, 
one can define $\tau_{j,\omega}\in T(B)$ by 
\[
\tau_{j,\omega}(\pi((a_n)_n))=\lim_{n\to\omega}\tau_{j,n}(a_n). 
\]
We obtain a probability measure $\mu_j$ on $X$ corresponding to 
$\tau_{j,\omega}\circ\pi\circ\tilde\phi
=\tau_{j,\omega}\circ\pi\circ\tilde\psi$. 
In the same way as in Lemma \ref{matrix}, 
we can find pairwise disjoint open subsets $W_1,W_2,\dots,W_l$ of $X$ 
such that 
\[
\max_j\mu_j(W_i)>0\quad \forall i=1,2,\dots,l,\quad 
\max_j\mu_j(X\setminus(W_1\cup W_2\cup\dots\cup W_l))=0
\]
and $\lvert f(y)-f(y')\rvert<\ep/3$ for any $y,y'\in W_i$ and $f\in F$. 
Choose $z_i\in W_i$ for each $i=1,2,\dots,l$. 
In the same way as in Lemma \ref{matrix}, 
we also get a family of mutually orthogonal non-zero projections 
$p_1,p_2,\dots,p_l$ in $B$ such that 
$\tau_{j,\omega}(p_i)>\mu_j(W_i)-\ep/2l$ 
and $\lVert\pi(\tilde\phi(f))p_i-f(z_i)p_i\rVert<\ep/3$ for all $f\in F$. 
Similarly one can find mutually orthogonal non-zero projections 
$q_1,q_2,\dots,q_l$ in $B$ for $\tilde\psi$. 
It is well-known that the projections $p_i$ (resp. $q_i$) lift to 
mutually orthogonal projections $(p_{i,n})_n$ (resp. $(q_{i,n})_n$) 
in $\prod A_n$. 
Then there exists $N\in\omega$ such that 
\[
p_{i,n}\neq0,\quad q_{i,n}\neq0,\quad 
\tau_{j,n}(p_{i,n})>\mu_j(W_i)-\ep/2l,\quad 
\tau_{j,n}(q_{i,n})>\mu_j(W_i)-\ep/2l
\]
holds for every $i=1,2,\dots,l$, $j=1,2,\dots,m$ and $n\in N$. 
For each $n\in\N$, the image of $D_{A_n}$ is dense in $\Aff(T(A_n))$ 
by \cite[Theorem 6.9.3]{Blacktext}. 
It follows that for each $n\in N$ and $i=1,2,\dots,l$ 
there exist projections $r_{i,n}\in A_n$ such that 
\[
\mu_j(W_i)-\ep/2l<\tau_{j,n}(r_{i,n})
<\min\{\tau_{j,n}(p_{i,n}),\tau_{j,n}(q_{i,n})\}\quad 
\forall j=1,2,\dots,m. 
\]
Besides $A_n$ satisfies strict comparison of projections 
(see Remark \ref{fourclasses} (2)). 
Therefore, for $n\in N$, we can find projections $p'_{i,n}\leq p_{i,n}$, 
$q'_{i,n}\leq q_{i,n}$ and a unitary $u_n\in A_n$ 
such that $\tau_{j,n}(p'_{i,n})>\mu_j(W_i)-\ep/2l$ and 
$p'_{i,n}=u_nq'_{i,n}u_n^*$. 
For $n\notin N$, set $p'_{i,n}=q'_{i,n}=0$ and $u_n=1$. 
Let $p'_i,q'_i,u\in B$ be the image of 
$(p'_{i,n})_n$, $(q'_{i,n})_n$ and $(u_n)_n$ by $\pi$. 
Then we have $p'_i\leq p_i$, $q'_i\leq q_i$, $p'_i=uq'_iu^*$ and 
$\tau_{j,\omega}(p'_i)>\mu_j(W_i)-\ep/l$. 
The rest of the proof is exactly the same as Lemma \ref{matrix}. 
\end{proof}

We can show the same statement for the case that 
the target algebra is of tracial rank zero. 

\begin{lem}\label{TAF}
Let $X$ be a compact metrizable space. 
For any finite subset $F\subset C(X)$ and $\ep>0$, 
there exist a finite subset $G\subset C(X)$ and $\delta>0$ 
such that the following hold. 
Let $A\in\mathcal{T}$ and 
let $\phi$ and $\psi$ be $(G,\delta)$-multiplicative ucp maps 
from $C(X)$ to $A$ such that 
\[
\lvert\tau(\phi(f))-\tau(\psi(f))\rvert<\delta\quad 
\forall f\in G,\ \tau\in T(A). 
\]
Then there exist a projection $p\in A$, 
$(F,\ep)$-multiplicative ucp maps $\phi',\psi':C(X)\to pAp$ and 
a unital homomorphism $\sigma:C(X)\to(1{-}p)A(1{-}p)$ 
with finite dimensional range such that 
$\phi\sim_{F,\ep}\phi'\oplus\sigma$, $\psi\sim_{F,\ep}\psi'\oplus\sigma$ 
and $\tau(p)<\ep$ for any $\tau\in T(A)$. 
\end{lem}
\begin{proof}
Suppose that we are given a finite subset $F\subset C(X)$ and $\ep>0$. 
Applying Lemma \ref{matrix} for $F$ and $\ep$, 
we obtain $G\subset C(X)$ and $\delta>0$. 
Let $A$ be a unital simple separable $C^*$-algebra with tracial rank zero 
and let $\phi$ and $\psi$ be $(G,\delta)$-multiplicative ucp maps 
from $C(X)$ to $A$ such that 
\[
\lvert\tau(\phi(f))-\tau(\psi(f))\rvert<\delta
\]
for any $f\in G$ and $\tau\in T(A)$. 
Since $A$ has tracial rank zero, 
there exist a sequence of projections $e_n\in A$, 
a sequence of finite dimensional subalgebras $B_n$ of $A$ with $1_{B_n}=e_n$ 
and a sequence of ucp maps $\pi_n:A\to B_n$ 
such that the following hold. 
\begin{itemize}
\item $\lVert[a,e_n]\rVert\to0$ as $n\to\infty$ for any $a\in A$. 
\item $\lVert\pi_n(a)-e_nae_n\rVert\to0$ as $n\to\infty$ for any $a\in A$. 
\item $\tau(1-e_n)\to0$ as $n\to\infty$ uniformly on $T(A)$. 
\end{itemize}
It is easy to see that 
$\pi_n\circ\phi$ and $\pi_n\circ\psi$ are $(G,\delta)$-multiplicative 
for sufficiently large $n\in\N$. 
We would like to show that 
\[
\lvert\tau(\pi_n(\phi(f)))-\tau(\pi_n(\psi(f)))\rvert<\delta
\]
holds for every $f\in G$, $\tau\in T(B_n)$ and sufficiently large $n\in\N$. 
To this end, we assume that there exist $\tau_n\in T(B_n)$ such that 
\[
\max_{f\in G}
\lvert\tau_n(\pi_n(\phi(f)))-\tau_n(\pi_n(\psi(f)))\rvert\geq\delta. 
\]
Let $\tau\in A^*$ be an accumulation point of $\tau_n\circ\pi_n$. 
Clearly $\tau$ is a tracial state of $A$ and 
$\lvert\tau(\phi(f))-\tau(\psi(f))\rvert\geq\delta$ for some $f\in G$, 
which is a contradiction. 

Hence, Lemma \ref{matrix} implies that, for sufficiently large $n\in\N$, 
there exist a projection $p_n\in B_n$, 
$(F,\ep)$-multiplicative ucp maps $\phi'_n,\psi'_n:C(X)\to p_nB_np_n$ and 
a unital homomorphism $\sigma_n:C(X)\to(e_n{-}p_n)B_n(e_n{-}p_n)$ such that 
\[
\pi_n\circ\phi\sim_{F,\ep}\phi'_n\oplus\sigma_n,\quad 
\pi_n\circ\psi\sim_{F,\ep}\psi'_n\oplus\sigma_n
\]
and $\tau(p_n)<\ep$ for any $\tau\in T(B_n)$. 
Therefore the proof is completed. 
\end{proof}

The following is taken from \cite[Theorem 3.1]{GL00IJM}. 
We remark that its origin is found in \cite{D}. 

\begin{thm}[{\cite[Theorem 3.1]{GL00IJM}}]\label{GongLin}
Let $X$ be a compact metrizable space. 
For any finite subset $F\subset C(X)$ and $\ep>0$, 
there exist a finite subset $G\subset C(X)$, $\delta>0$, $l\in\N$ and 
a finite subset $L\subset\mathcal{K}(C(X))$ satisfying the following:
For any unital $C^*$-algebra $A$ with real rank zero, stable rank one and 
weakly unperforated $K_0(A)$ and 
any $(G,\delta)$-multiplicative ucp maps $\phi,\psi:C(X)\to A$ satisfying 
$\phi_\#|L=\psi_\#|L$, 
there exist a unitary $u\in M_{l+1}(A)$ and $\{x_1,x_2,\dots,x_l\}\subset X$ 
such that 
\[
\lVert u\diag(\phi(f),f(x_1),f(x_2),\dots,f(x_l))u^*
-\diag(\psi(f),f(x_1),f(x_2),\dots,f(x_l))\rVert<\ep
\]
for any $f\in F$. 
\end{thm}

The following theorem is a variant of \cite[Theorem 4.6]{L07Trans}. 

\begin{thm}\label{appunique1}
Let $X$ be a compact metrizable space, 
let $F\subset C(X)$ be a finite subset and let $\ep>0$. 
Then there exist a finite subset $L\subset\mathcal{K}(C(X))$ and 
a family of mutually orthogonal positive elements 
$h_1,h_2,\dots,h_k\in C(X)$ of norm one 
such that the following holds. 
For any $\nu>0$, 
one can find a finite subset $G\subset C(X)$ and $\delta>0$ 
satisfying the following. 
For any $A\in\mathcal{T}$ and 
any $(G,\delta)$-multiplicative ucp maps $\phi,\psi:C(X)\to A$ 
such that $\phi_\#|L=\psi_\#|L$, 
\[
\tau(\phi(h_i))\geq\nu\quad 
\forall \tau\in T(A),\ i{=}1,2,\dots,k
\]
and 
\[
\lvert\tau(\phi(f))-\tau(\psi(f))\rvert<\delta\quad 
\forall \tau\in T(A),\ f\in G, 
\]
there exists a unitary $u\in A$ such that 
\[
\lVert u\phi(f)u^*-\psi(f)\rVert<\ep
\]
holds for any $f\in F$. 
\end{thm}
\begin{proof}
We say that a subset $Y\subset X$ is $(F,\ep)$-dense 
if for any $x\in X$ there exists $y\in Y$ such that 
$\lvert f(x)-f(y)\rvert<\ep$ for every $f\in F$. 
Choose an $(F,\ep/7)$-dense finite subset $\{y_1,y_2,\dots,y_k\}\subset X$. 
For each $i=1,2,\dots,k$, choose an open neighborhood $U_i$ of $y_i$ so that 
$x\in U_i$ implies $\lvert f(x)-f(y_i)\rvert<\ep/7$ 
for any $f\in F$ and that 
$U_1,U_2,\dots,U_k$ are mutually disjoint. 
Choose a positive function $h_i\in C_0(U_i)$ of norm one. 
By applying Theorem \ref{GongLin} to $F$ and $\ep/7$, 
we obtain a finite subset $G_1\subset C(X)$, $\delta_1>0$, $l\in\N$ 
and a finite subset $L\subset\mathcal{K}(C(X))$. 
There exist a finite subset $G_2\subset C(X)$ and $\delta_2>0$ 
such that the following holds: 
For any unital $C^*$-algebra $A$ and 
any $(G_2,\delta_2)$-multiplicative ucp maps $\phi,\psi:C(X)\to A$, 
if $\lVert\phi(f)-\psi(f)\rVert<\delta_2$ for every $f\in G_2$, 
then $\phi_\#|L=\psi_\#|L$. 
Suppose that $\nu>0$ is given. 
Let 
\[
G_3=F\cup G_1\cup G_2\cup\{h_1,h_2,\dots,h_k\}
\]
and 
\[
\delta_3=\min\{\ep/7,\delta_1,\delta_2,\nu/(l{+}2)\}. 
\]
By applying Lemma \ref{TAF} to $G_3$ and $\delta_3$, 
we obtain a finite subset $G\subset C(X)$ and $\delta>0$. 

Suppose that $A$ is a unital simple separable $C^*$-algebra $A$ 
with tracial rank zero 
and that $\phi,\psi:C(X)\to A$ are $(G,\delta)$-multiplicative ucp maps 
satisfying $\phi_\#|L=\psi_\#|L$, 
\[
\tau(\phi(h_i))\geq\nu\quad 
\forall \tau\in T(A),\ i{=}1,2,\dots,k
\]
and 
\[
\lvert\tau(\phi(f))-\tau(\psi(f))\rvert<\delta\quad 
\forall \tau\in T(A),\ f\in G. 
\]
By lemma \ref{TAF}, 
there exist a projection $p\in A$, 
$(G_3,\delta_3)$-multiplicative ucp maps $\phi',\psi':C(X)\to pAp$, 
a unital homomorphism $\sigma:C(X)\to(1{-}p)A(1{-}p)$ 
with finite dimensional range such that 
$\phi\sim_{G_3,\delta_3}\phi'\oplus\sigma$, 
$\psi\sim_{G_3,\delta_3}\psi'\oplus\sigma$ 
and $\tau(p)<\delta_3$ for any $\tau\in T(A)$. 
Since $G_2$ is contained in $G_3$ and 
$\delta_2$ is not greater than $\delta_3$, 
by the choice of $G_2$ and $\delta_2$, 
we obtain $(\phi'\oplus\sigma)_\#|L=(\psi'\oplus\sigma)_\#|L$, 
and hence $\phi'_\#|L=\psi'_\#|L$. 
Besides, $\phi'$ and $\psi'$ are $(G_1,\delta_1)$-multiplicative, 
because $G_1$ is contained in $G_3$ and 
$\delta_1$ is not greater than $\delta_3$. 
By Theorem \ref{GongLin}, 
there exist a unitary $u\in M_{l+1}(pAp)$ and 
$\{x_1,x_2,\dots,x_l\}\subset X$ such that 
\[
\lVert u\diag(\phi'(f),f(x_1),f(x_2),\dots,f(x_l))u^*
-\diag(\psi'(f),f(x_1),f(x_2),\dots,f(x_l))\rVert<\ep/7
\]
for any $f\in F$. 
In what follows, for a positive linear functional $\rho$ on $C(X)$, 
we let $\mu_\rho$ denote the corresponding measure on $X$. 
For any $\tau\in T(A)$ and $i=1,2,\dots,k$, one has 
\[
\mu_{\tau\circ\sigma}(U_i)\geq\tau(\sigma(h_i))
>\tau(\phi(h_i))-\tau(\phi'(h_i))-\delta_3
>\nu-2\delta_3
\geq l\delta_3. 
\]
It follows that 
there exists a unital homomorphism $\sigma':C(X)\to(1{-}p)A(1{-}p)$ 
with finite dimensional range such that 
\[
\lVert\sigma(f)-\sigma'(f)\rVert<\ep/7
\]
for any $f\in F$ and 
$\mu_{\tau\circ\sigma'}(\{y_i\})>l\delta_3$ 
for any $\tau\in T(A)$ and $i=1,2,\dots,k$. 
Since $\{y_1,y_2,\dots,y_k\}$ is $(F,\ep/7)$-dense, 
we can find a unital homomorphism $\sigma'':C(X)\to(1{-}p)A(1{-}p)$ 
with finite dimensional range such that 
\[
\lVert\sigma'(f)-\sigma''(f)\rVert<\ep/7
\]
for any $f\in F$ and $\mu_{\tau\circ\sigma''}(\{x_j\})>\delta_3$ 
for any $\tau\in T(A)$ and $j=1,2,\dots,l$. 
Then it is not so hard to see 
$\phi'\oplus\sigma''\sim_{F,\ep/7}\psi'\oplus\sigma''$. 
Consequently we have 
\begin{align*}
\phi&\sim_{F,\ep/7}\phi'\oplus\sigma\sim_{F,\ep/7}\phi'\oplus\sigma'
\sim_{F,\ep/7}\phi'\oplus\sigma'' \\
&\sim_{F,\ep/7}\psi'\oplus\sigma'' 
\sim_{F,\ep/7}\psi'\oplus\sigma'
\sim_{F,\ep/7}\psi'\oplus\sigma
\sim_{F,\ep/7}\psi. 
\end{align*}
\end{proof}

In the same fashion as above, one can prove the following 
by using Lemma \ref{rr0} instead of Lemma \ref{TAF} 
(see also \cite[Corollary 2.17]{GL00ActaSin}). 

\begin{thm}\label{appunique2}
Let $X$ be a compact metrizable space, 
let $F\subset C(X)$ be a finite subset and let $\ep>0$, $m\in\N$. 
Then there exist a finite subset $L\subset\mathcal{K}(C(X))$ and 
a family of mutually orthogonal positive elements 
$h_1,h_2,\dots,h_k\in C(X)$ of norm one 
such that the following holds. 
For any $\nu>0$, 
one can find a finite subset $G\subset C(X)$ and $\delta>0$ 
satisfying the following. 
Let $A\in\mathcal{T}'$ be a $C^*$-algebra 
with at most $m$ extremal tracial states. 
For any $(G,\delta)$-multiplicative ucp maps $\phi,\psi:C(X)\to A$ 
such that $\phi_\#|L=\psi_\#|L$, 
\[
\tau(\phi(h_i))\geq\nu\quad 
\forall \tau\in T(A),\ i{=}1,2,\dots,k
\]
and 
\[
\lvert\tau(\phi(f))-\tau(\psi(f))\rvert<\delta\quad 
\forall \tau\in T(A),\ f\in G, 
\]
there exists a unitary $u\in A$ such that 
\[
\lVert u\phi(f)u^*-\psi(f)\rVert<\ep
\]
holds for any $f\in F$. 
\end{thm}

By using the theorems above, 
we obtain the following generalization of \cite[Theorem 3.3]{L07Trans}. 

\begin{thm}\label{appunique3}
Let $X$ be a compact metrizable space and 
let $A\in\mathcal{T}\cup\mathcal{T}'$. 
Let $\phi:C(X)\to A$ be a unital monomorphism. 
Then for any finite subset $F\subset C(X)$ and $\ep>0$, 
there exist a finite subset $L\subset\mathcal{K}(C(X))$, 
a finite subset $G\subset C(X)$ and $\delta>0$ 
such that the following hold. 
If $\psi:C(X)\to A$ is a $(G,\delta)$-multiplicative ucp map satisfying 
$\phi_\#|L=\psi_\#|L$ and 
\[
\lvert\tau(\phi(f))-\tau(\psi(f))\rvert<\delta
\]
for any $\tau\in T(A)$ and $f\in G$, then 
there exists a unitary $u\in A$ such that 
\[
\lVert u\phi(f)u^*-\psi(f)\rVert<\ep
\]
holds for any $f\in F$. 
\end{thm}
\begin{proof}
Applying Theorem \ref{appunique1} or Theorem \ref{appunique2}, 
we obtain a finite subset $L\subset\mathcal{K}(C(X))$ and 
positive elements $h_1,h_2,\dots,h_k\in C(X)$ of norm one. 
Since $A$ is simple and $\phi$ is injective, 
\[
\nu=\min\{\tau(\phi(h_i))\mid\tau\in T(A),\ i=1,2,\dots,k\}
\]
is positive. 
Using Theorem \ref{appunique1} or Theorem \ref{appunique2} for $\nu$, 
we find a finite subset $G\subset C(X)$ and $\delta>0$. 
It is clear that $G$ and $\delta$ meet the requirement. 
\end{proof}

The following is an immediate consequence of the theorem above. 

\begin{thm}\label{rr0unique}
Let $X$ be a compact metrizable space and 
let $A\in\mathcal{T}\cup\mathcal{T}'$. 
Let $\phi,\psi:C(X)\to A$ be unital monomorphisms. 
Then $\phi$ and $\psi$ are approximately unitarily equivalent 
if and only if $KL(\phi)=KL(\psi)$ and 
$\tau\circ\phi=\tau\circ\psi$ for all $\tau\in T(A)$. 
\end{thm}

\begin{cor}\label{rr0uniqueAH}
Let $C$ be a unital AH algebra and 
let $A\in\mathcal{T}\cup\mathcal{T}'$. 
Let $\phi,\psi:C\to A$ be unital monomorphisms. 
Then $\phi$ and $\psi$ are approximately unitarily equivalent 
if and only if $KL(\phi)=KL(\psi)$ and 
$\tau\circ\phi=\tau\circ\psi$ for all $\tau\in T(A)$. 
\end{cor}
\begin{proof}
Although the proof is essentially the same as \cite[Corollary 4.8]{L07Trans}, 
we present it for completeness. 
Without loss of generality, 
we may assume $C=p(C(X)\otimes M_k)p$, 
where $X$ is a compact metrizable space and 
$p\in C(X)\otimes M_k$ is a non-zero projection. 
We may further assume that 
the rank of $p(x)\in M_k$ is strictly positive for every $x\in X$. 

We first consider the case $p=1\in C(X)\otimes M_k$. 
It is easy to see that there exists a unitary $u\in A$ such that 
$\phi(1\otimes a)=u\psi(1\otimes a)u^*$ holds for any $a\in M_k$. 
Let $e$ be a minimal projection of $M_k$. 
Then $\phi'(f)=\phi(f\otimes e)$ and $\psi'(f)=u\psi(f\otimes e)u^*$ 
are unital monomorphisms from $C(X)$ to $\phi(e)A\phi(e)$. 
By Theorem \ref{rr0unique}, they are approximately unitarily equivalent. 
Hence $\phi$ and $\psi$ are approximately unitarily equivalent. 

Let us consider the general case. 
There exist $l\in\N$ and 
a projection $q\in C\otimes M_l\subset C(X)\otimes M_{kl}$ such that 
$p\otimes e$ is a subprojection of $q$ and $q$ is 
Murray-von Neumann equivalent to $1_{C(X)}\otimes r$, 
where $e\in M_l$ is a minimal projection of $M_l$ and 
$r\in M_{kl}$ is a projection of rank $k$. 
By the argument above, 
the restrictions of $\phi\otimes\id_{M_l}$ and $\psi\otimes\id_{M_l}$ 
to $q(C\otimes M_l)q$ are approximately unitarily equivalent. 
It follows that 
their restrictions to $(p\otimes e)(C\otimes M_l)(p\otimes e)=C$ 
are also approximately unitarily equivalent, 
which completes the proof. 
\end{proof}

In Section6 we will generalize the results above to the case 
that the target algebra $A$ belongs to $\mathcal{C}\cup\mathcal{C}'$.

%%%%%%%%%%%%%%%%%%%%%%%%%%%%%%%%%%%%%%%%%%%%%%%%%%%%%%%%%%%%
\section{Homotopy of unitaries}

In this section, we prove the so called basic homotopy lemma 
for $A$ in $\mathcal{T}\cup\mathcal{T}'$ 
(Theorem \ref{Basic} and Theorem \ref{Basic2}). 
The basic idea of the proof is similar to that of \cite[Theorem 8.1]{L0612}, 
but there are two main differences. 
One is the use of Theorem \ref{NWexist}, 
which claims the existence of a unital monomorphism $\phi:C(X)\to A$ 
realizing the given $\kappa\in KL(C(X),A)_{+,1}$ and 
$\lambda:T(A)\to T(C(X))$. 
The other point is that 
we allow $G\subset C(X)$ and $\delta>0$ in Theorem \ref{Basic} 
to depend on the given homomorphism $\phi:C(X)\to A$. 
Although, as shown in \cite{L0612}, 
it is possible to state the theorem in a more general form, 
we do not pursue this here 
because the actual application discussed in Section 6 does not need 
that general form. 
These two points enable us to simplify the proof given in \cite{L0612}. 

We let $\T$ denote the unit circle in the complex plane and 
let $z\in C(\T)$ be the identity function 
$z(\exp(\pi\sqrt{-1}t))=\exp(\pi\sqrt{-1}t)$. 
The following is a variant of \cite[Lemma 6.4]{L0612}. 

\begin{lem}\label{fullspec}
Let $X$ be a compact metrizable space and 
let $A\in\mathcal{T}\cup\mathcal{T}'$. 
For any finite subsets $F\subset C(X)$, $\widetilde F\subset C(X\times\T)$ 
and $\ep>0$, 
there exist a finite subset $G\subset C(X)$ and $\delta>0$ 
such that the following hold. 
For any $k\in\N$, 
any unital monomorphism $\phi:C(X)\to A$ and a unitary $u\in A$ 
satisfying $\lVert[\phi(f),u]\rVert<\delta$ for any $f\in G$, 
there exist a path of unitaries $w:[0,1]\to A$ 
and an $(\widetilde F,\ep)$-multiplicative ucp map $\psi:C(X\times\T)\to A$ 
such that 
\[
\lVert w(0)-u\rVert<\ep,\quad 
\lVert w(1)-\psi(1\otimes z)\rVert<\ep,\quad \Lip(w)\leq\pi, 
\]
\[
\lVert[\phi(f),w(t)]\rVert<\ep,\quad 
\lVert\psi(f\otimes1)-\phi(f)\rVert<\ep
\]
hold for any $f\in F$ and $t\in[0,1]$, and 
\[
\lvert\tau(\psi(f\otimes z^j))\rvert<\ep\lVert f\rVert
\]
holds for any $\tau\in T(A)$, $f\in C(X)$ and $j\in\Z$ 
with $1\leq\lvert j\rvert<k$. 
\end{lem}
\begin{proof}
Without loss of generality, 
we may assume that all the elements of $F$ are of norm one. 
Applying Lemma \ref{rr0} or Lemma \ref{TAF} to 
\[
G_1=\widetilde F\cup
\{f\otimes1\mid f\in F\}\cup\{1\otimes z\}\subset C(X\times\T)
\]
and $\delta_1=\min\{\ep/8,\ep^2\}$, 
we obtain a finite subset $G_2\subset C(X\times\T)$ and $\delta_2>0$. 
We may assume that $G_2$ contains $G_1$ and 
that $\delta_2$ is less than $\delta_1$. 
Clearly there exist a finite subset $G\subset C(X)$ and $\delta>0$ 
such that the following holds: 
If $\phi:C(X)\to A$ is a unital monomorphism and $u\in A$ is a unitary 
satisfying $\lVert[\phi(f),u]\rVert<\delta$ for any $f\in G$, 
then one can find a $(G_2,\delta_2)$-multiplicative ucp map 
$\phi_0:C(X\times\T)\to A$ such that 
\[
\lVert\phi_0(1\otimes z)-u\rVert<\delta_2,\quad 
\lVert\phi_0(f\otimes1)-\phi(f)\rVert<\delta_2
\]
for every $f\in G_2$. 

Suppose that we are given $k\in\N$, 
a unital monomorphism $\phi:C(X)\to A$ and a unitary $u\in A$ 
satisfying $\lVert[\phi(f),u]\rVert<\delta$ for every $f\in G$. 
We find $\phi_0:C(X\times\T)\to A$ as above. 
By using Lemma \ref{rr0} or Lemma \ref{TAF}, 
there exist a projection $p\in A$, 
a $(G_1,\delta_1)$-multiplicative ucp map $\phi_0':C(X\times\T)\to pAp$ and 
a unital homomorphism $\sigma:C(X\times\T)\to(1{-}p)A(1{-}p)$ 
with finite dimensional range such that 
\[
\lVert\phi_0(f)-(\phi_0'\oplus\sigma)(f)\rVert<\delta_1\quad 
\forall f\in G_1
\]
and $\tau(p)<\delta_1$ for any $\tau\in T(A)$. 
We may further assume that 
there exists a unitary $u'\in pAp$ such that 
$\lVert u'-\phi_0'(1\otimes z)\rVert<\delta_1$. 
Since $\sigma$ has finite dimensional range, 
one can find $x_1,x_2,\dots,x_l\in X$, $y_1,y_2,\dots,y_l\in\T$ and 
projections $p_1,p_2,\dots,p_l\in A$ such that 
\[
\sum_{i=1}^lp_i=1{-}p,\quad 
\sigma(f\otimes g)=\sum_{i=1}^lf(x_i)g(y_i)p_i
\]
holds for any $f\in C(X)$ and $g\in C(\T)$. 
By replacing each $p_i$ with its subprojection if necessary, 
we may assume that $D_A([p_i])$ belongs to $k\Ima(D_A)$, 
because $\Ima D_A$ is dense in $\Aff(T(A))$. 
Choose projections $q_{i,j}$ for $i=1,2,\dots,l$ and $j=1,2,\dots,k$ 
so that 
\[
\sum_{j=1}^kq_{i,j}=p_i,\quad kD_A([q_{i,j}])=D_A([p_i])\quad 
\forall j=1,2,\dots,k. 
\]
Define a homomorphism $\sigma':C(X\times\T)\to(1{-}p)A(1{-}p)$ 
with finite dimensional range by 
\[
\sigma'(f\otimes g)=\sum_{i=1}^l\sum_{j=1}^kf(x_i)g(\zeta^j)q_{i,j}, 
\]
where $\zeta=\exp(2\pi\sqrt{-1}/k)$. 
Define a ucp map $\psi:C(X\times\T)\to A$ by $\psi=\phi_0'\oplus\sigma'$. 
It is clear that $\psi$ is $(G_1,\delta_1)$-multiplicative, 
and hence is $(\widetilde F,\ep)$-multiplicative. 
Moreover, one has 
\[
\psi(f\otimes1)=\phi_0'(f\otimes1)\oplus\sigma'(f\otimes1)
=\phi_0'(f\otimes1)\oplus\sigma(f\otimes1)
\approx_{\delta_1}\phi_0(f\otimes1)\approx_{\delta_2}\phi(f)
\]
for any $f\in F$. 
For any $\tau\in T(A)$, $f\in C(X)$ and $j\in\Z$ with $1\leq\lvert j\rvert<k$, 
it is easy to see 
\begin{align*}
\lvert\tau(\psi(f\otimes z^j))\rvert
&\leq\lvert\tau(\phi_0'(f\otimes z^j))\lvert
+\lvert\tau(\sigma'(f\otimes z^j))\rvert \\
&=\lvert\tau(\phi_0'(f\otimes z^j))\lvert \\
&\leq\lVert f\rVert\tau(p)^{1/2}<\ep\lVert f\rVert. 
\end{align*}
We construct a path of unitaries $w:[0,1]\to A$. 
By the definition of $\sigma'$, 
we can find a path of unitaries $v:[0,1]\to(1{-}p)A(1{-}p)$ such that 
\[
v(0)=\sigma(1\otimes z),\quad v(1)=\sigma'(1\otimes z),\quad 
\Lip(v)\leq\pi
\]
and $[q_{i,j},v(t)]=0$ for any $i,j$ and $t\in[0,1]$. 
Define $w:[0,1]\to U(A)$ by $w(t)=u'\oplus v(t)$. 
Evidently we have 
\[
w(0)=u'\oplus\sigma(1\otimes z)
\approx_{\delta_1}\phi_0'(1\otimes z)\oplus\sigma(1\otimes z)
\approx_{\delta_1}\phi_0(1\otimes z)\approx_{\delta_2}u, 
\]
\[
w(1)=u'\oplus\sigma'(1\otimes z)
\approx_{\delta_1}\phi_0'(1\otimes z)\oplus\sigma'(1\otimes z)
=\psi(1\otimes z). 
\]
and $\Lip(w)\leq\pi$. 
Besides, for any $f\in F$ and $t\in[0,1]$, one can verify 
\begin{align*}
[\phi(f),w(t)]&\approx_{2\delta_2}[\phi_0(f\otimes1),w(t)]
\approx_{2\delta_1}[(\phi_0'\oplus\sigma)(f\otimes1),u'\oplus v(t)]\\
&=[\phi_0'(f\otimes1),u']
\approx_{2\delta_1}[\phi_0'(f\otimes1),\phi_0'(1\otimes z)]
\approx_{2\delta_1}0, 
\end{align*}
thereby completing the proof. 
\end{proof}

\begin{rem}
In the lemma above, for $A$ in $\mathcal{T}$, one can see that 
$G\subset C(X)$ and $\delta$ depend only on $F\subset C(X)$ and $\ep$. 
For $A$ in $\mathcal{T}'$, 
$G\subset C(X)$ and $\delta$ depend only on 
$F\subset C(X)$, $\ep$ and the cardinality of extremal tracial states on $A$. 
\end{rem}

The following is a generalization of \cite[Corollary 8.4]{L0612}. 

\begin{thm}\label{Basic}
Let $X$ be a path connected compact metrizable space and 
let $A\in\mathcal{T}\cup\mathcal{T}'$. 
Let $\phi:C(X)\to A$ be a unital monomorphism. 
For any finite subset $F\subset C(X)$ and $\ep>0$, 
there exist a finite subset $L\subset\mathcal{K}(C(X))$, 
a finite subset $G\subset C(X)$ and $\delta>0$ 
such that the following hold. 
If $u\in A$ is a unitary satisfying 
\[
\lVert[\phi(f),u]\rVert<\delta\quad \forall f\in G\quad\text{and}\quad 
\Bott(\phi,u)(x)=0\quad \forall x\in L, 
\]
then there exists a path of unitaries $w:[0,1]\to A$ such that 
\[
w(0)=u,\quad w(1)=1,\quad \Lip(w)<2\pi+\ep
\]
and 
\[
\lVert[\phi(f),w(t)]\rVert<\ep\quad \forall f\in F,\ t\in[0,1]. 
\]
\end{thm}
\begin{proof}
Let $p:C([-1,1])\to\C$ be the point evaluation at $1\in[-1,1]$ and 
let $q:C(\T)\to C([-1,1])$ be the unital monomorphism 
defined by $q(f)(t)=f(\exp(\pi\sqrt{-1}t))$. 

By Lemma \ref{KL+1}, 
$KL(\phi)\circ KL(\id\otimes p)$ belongs to $KL(C(X\times[-1,1]),A)_{+,1}$. 
Define $\tau_0\in T(C([-1,1]))$ by 
\[
\tau_0(f)=\frac{1}{2}\int_{-1}^1f(t)\,d\mu(t), 
\]
where $\mu$ is the Lebesgue measure on $\R$. 
Let $\lambda:T(C(X))\to T(C(X\times[-1,1]))$ be the affine continuous map 
defined by $\lambda(\tau)=\tau\otimes\tau_0$. 
By applying Theorem \ref{NWexist} to 
$KL(\phi)\circ KL(\id\otimes p)$ and $\lambda\circ T(\phi)$, 
we get a unital monomorphism $\sigma:C(X\times[-1,1])\to A$ such that 
$KL(\sigma)=KL(\phi)\circ KL(\id\otimes p)$ and 
$T(\sigma)=\lambda\circ T(\phi)$. 
Then $\sigma'=\sigma\circ(\id\otimes q)$ is 
a unital monomorphism from $C(X\times\T)$ to $A$ such that 
\[
KL(\sigma')=KL(\sigma\circ(\id\otimes q))
=KL(\phi)\circ KL(\id\otimes(p\circ q))
\]
and $T(\sigma')=T(\id\otimes q)\circ\lambda\circ T(\phi)$. 
Under the canonical isomorphism 
\[
\underline{K}(C(X\times\T))
\cong\underline{K}(C(X))\oplus\underline{K}(C_0(X\times(\T\setminus\{-1\}))), 
\]
$KL(\sigma')\in\Hom_\Lambda(\underline{K}(C(X\times\T)),\underline{K}(A))$ 
corresponds to $KL(\phi)\oplus0$. 
It is also easy to see that 
$T(\sigma')(\tau)=(\tau\circ\phi)\otimes\tau_0'$ for any $\tau\in T(A)$, 
where $\tau_0'\in T(C(\T))$ is the tracial state 
corresponding to the Haar measure on $\T$. 
From the construction, 
there exists a path of unitaries $w_1:[0,1]\to A$ such that 
\[
w_1(0)=1,\quad w_1(1)=\sigma'(1\otimes z),\quad \Lip(w_1)=\pi
\]
and $[\sigma'(f\otimes1),w_1(t)]=0$ for any $f\in C(X)$ and $t\in[0,1]$. 

By applying Theorem \ref{appunique3} to 
$\sigma':C(X\times\T)\to A$, $\{f\otimes1\mid f\in F\}\cup\{1\otimes z\}$ and 
$\ep/4$, we obtain a finite subset $L\subset\mathcal{K}(C(X\times\T))$, 
a finite subset $G_1\subset C(X\times\T)$ and $\delta_1>0$. 
Choose a sufficiently large finite subset $L_0\subset\mathcal{K}(C(X))$, 
a sufficiently large finite subset $G_2\subset C(X)$ and 
a sufficiently small real number $\delta_2>0$. 
By applying Lemma \ref{fullspec} to 
$G_2\subset C(X)$, $G_1\subset C(X\times\T)$ and $\delta_2>0$, 
we obtain a finite subset $G\subset C(X)$ and $\delta>0$. 

Suppose that we are given a unitary $u\in A$ satisfying 
\[
\lVert[\phi(f),u]\rVert<\delta\quad \forall f\in G\quad\text{and}\quad 
\Bott(\phi,u)(x)=0\quad \forall x\in L_0. 
\]
Let $k\in\N$ be a sufficiently large natural number. 
By Lemma \ref{fullspec}, 
one can find a path of unitaries $w_0:[0,1]\to A$ 
and a $(G_1,\delta_2)$-multiplicative ucp map $\psi:C(X\times\T)\to A$ 
such that 
\[
\lVert w_0(0)-u\rVert<\delta_2,\quad 
\lVert w_0(1)-\psi(1\otimes z)\rVert<\delta_2,\quad \Lip(w_0)\leq\pi, 
\]
\[
\lVert[\phi(f),w_0(t)]\rVert<\delta_2,\quad 
\lVert\psi(f\otimes1)-\phi(f)\rVert<\delta_2
\]
hold for any $f\in G_2$ and $t\in[0,1]$, and 
\[
\lvert\tau(\psi(f\otimes z^j))\rvert<\delta_2\lVert f\rVert
\]
holds for any $\tau\in T(A)$, $f\in C(X)$ and $j\in\Z$ 
with $1\leq\lvert j\rvert<k$. 
Hence, if $L_0\subset\mathcal{K}(C(X))$ is large enough, 
$G_2\subset C(X)$ is large enough and $\delta_2>0$ is small enough, 
then one can conclude $\psi_\#|L=\sigma'_\#|L$. 
In addition, if $k\in\N$ is chosen to be large enough, then we may assume 
\[
\lvert\tau(\psi(f))-\tau(\sigma'(f))\rvert<\delta_1\quad 
\forall \tau\in T(A),\ f\in G_1. 
\]
It follows from Theorem \ref{appunique3} that 
there exists a unitary $v\in A$ such that 
\[
\lVert v\sigma'(1\otimes z)v^*-\psi(1\otimes z)\Vert<\ep/4,\quad 
\lVert v\sigma'(f\otimes1)v^*-\psi(f\otimes1)\Vert<\ep/4\quad \forall f\in F. 
\]
We define $w:[0,1]\to U(A)$ by $w(t)=w_0(t)vw_1(t)^*v^*$. 
Clearly one has $\Lip(w)\leq2\pi$, 
\[
\lVert w(0)-u\rVert<\delta_2,\quad \lVert w(1)-1\rVert<\delta_2+\ep/4
\]
and 
\[
\lVert[\phi(f),w(t)]\rVert<3\delta_2+\ep/2\quad 
\forall f\in F,\ t\in[0,1]. 
\]
It is easy to perturb the path $w:[0,1]\to A$ a little bit 
so that $w(0)=u$ and $w(1)=1$. 
\end{proof}

The following is an easy generalization of the theorem above. 

\begin{thm}\label{Basic2}
Let $C$ be a unital $C^*$-algebra of the form 
$\bigoplus_{i=1}^np_iM_{k_i}(C(X_i))p_i$, 
where $X_i$ is a path connected compact metrizable space and 
$p_i$ is a non-zero projection of $M_{k_i}(C(X_i))$. 
Let $A\in\mathcal{T}\cup\mathcal{T}'$. 
Let $\phi:C\to A$ be a unital monomorphism. 
For any finite subset $F\subset C$ and $\ep>0$, 
there exist a finite subset $L\subset\mathcal{K}(C)$, 
a finite subset $G\subset C$ and $\delta>0$ 
such that the following hold. 
If $u\in A$ is a unitary satisfying 
\[
\lVert[\phi(f),u]\rVert<\delta\quad \forall f\in G\quad\text{and}\quad 
\Bott(\phi,u)(x)=0\quad \forall x\in L, 
\]
then there exists a path of unitaries $w:[0,1]\to A$ such that 
\[
w(0)=u,\quad w(1)=1,\quad \Lip(w)<2\pi+\ep
\]
and 
\[
\lVert[\phi(f),w(t)]\rVert<\ep\quad \forall f\in F,\ t\in[0,1]. 
\]
\end{thm}
\begin{proof}
We can prove this in a similar fashion to \cite[Lemma 17.5]{L0612} 
by using the theorem above. 
We omit the detail. 
It is worth noting that if $A$ is in $\mathcal{T}\cup\mathcal{T}'$ and 
$e\in A$ is a non-zero projection, 
then $eAe$ is also in $\mathcal{T}\cup\mathcal{T}'$. 
See also the proof of Corollary \ref{rr0uniqueAH}. 
\end{proof}

\begin{rem}\label{Qstable}
In the theorems above, 
if the target algebra $A$ satisfies $A\cong A\otimes Q$ 
(i.e. $A$ is $Q$-stable), then 
$K_i(A;\Z_n)=0$ for any $i=0,1$ and $n\geq1$ 
because $K_i(A)$ is torsion free and divisible. 
Therefore the entire $K$-group $\underline{K}(A)$ is 
canonically isomorphic to $K_0(A)\oplus K_1(A)$. 
Consequently we may assume that 
the finite subset $L\subset\mathcal{K}(C(X))$ in the statement is 
actually a finite subset of $P(C(X)\otimes\K)\cup U_\infty(C(X))$. 
\end{rem}

%%%%%%%%%%%%%%%%%%%%%%%%%%%%%%%%%%%%%%%%%%%%%%%%%%%%%%%%%%%%
\section{$\mathcal{Z}$-stable $C^*$-algebras}

In this section we prove Theorem \ref{main} and Corollary \ref{mainAH}.  
When $X$ is a finite CW complex, 
it is well-known that $K_*(C(X))$ is finitely generated. 

\begin{lem}\label{adjust1}
Let $C$ be a $C^*$-algebra of the form $p(C(X)\otimes M_k)p$, 
where $X$ is a finite CW complex and $p\in C(X)\otimes M_k$ is a projection. 
Let $A\in\mathcal{T}\cup\mathcal{T}'$. 
Let $L\subset U_\infty(C)$ be a finite subset which generates $K_1(C)$ 
and let $\phi:C\to A$ be a unital monomorphism. 
For any finite subset $F\subset C$ and $\ep>0$, 
there exists $\delta>0$ such that the following holds. 
If $\xi:K_1(C)\to K_0(A)$ satisfies 
$\lVert D_A(\xi([w]))\rVert<\delta$ for any $w\in L$, 
then there exists a unitary $u\in A$ such that 
\[
\lVert[\phi(f),u]\rVert<\ep
\]
for every $f\in F$ and 
\[
\Bott(\phi,u)(w)=\xi([w])
\]
for every $w\in L$. 
\end{lem}
\begin{proof}
When $A$ is in $\mathcal{T}$, 
this lemma is contained in \cite[Lemma 6.11]{L09AJM}. 
Assume $A\in\mathcal{T}'$. 
By Theorem \ref{LinK}, 
there exist a unital simple AH algebra $B$ with real rank zero 
and slow dimension growth and 
a unital homomorphism $\psi:B\to A$ such that 
$K_*(\psi)$ gives a graded ordered isomorphism. 
The tracial simplexes $T(A)$ and $T(B)$ are naturally 
isomorphic to the state spaces of $K_0(A)$ and $K_0(B)$, respectively. 
Hence $T(\psi)$ induces an affine isomorphism from $T(A)$ to $T(B)$. 
It follows from Corollary \ref{NWexist2} that 
there exists a unital homomorphism $\phi_0:C\to B$ such that 
$KL(\phi_0)=KL(\psi)^{-1}\circ KL(\phi)$ and 
$T(\phi_0)=T(\phi)\circ T(\psi)^{-1}$. 
By Corollary \ref{rr0uniqueAH}, $\psi\circ\phi_0$ and $\phi$ are 
approximately unitarily equivalent. 
As $B$ is in $\mathcal{T}$, 
we have already known that the lemma holds for $\phi_0:C\to B$. 
Therefore the lemma holds for $\psi\circ\phi_0:C\to A$, 
and hence for $\phi:C\to A$. 
\end{proof}

\begin{rem}
In the lemma above the finite subset $L\subset U_\infty(C)$ is allowed to be 
\textit{any} finite subset which generates $K_1(C)$, 
though this point is not clearly mentioned in \cite[Lemma 6.11]{L09AJM}. 
This readily follows from the fact that 
(if $F$ is large enough and $\ep$ is small enough, then) 
$\Bott(\phi,u)$ gives rise to a `partial homomorphism' 
from $K_1(C)$ to $K_0(A)$, as mentioned in Section 2.3. 
\end{rem}

In what follows, 
we frequently omit `$\otimes\id$', `$\otimes1$' and `$\otimes\Tr$' 
to simplify notation. 
For example, $u\otimes1\in A\otimes M_n$ is denoted by $u$. 

\begin{lem}\label{adjust2}
Let $C$ be a $C^*$-algebra of the form $p(C(X)\otimes M_k)p$, 
where $X$ is a finite CW complex and $p\in C(X)\otimes M_k$ is a projection. 
Let $A\in\mathcal{T}\cup\mathcal{T}'$. 
Suppose that 
unital monomorphisms $\phi,\psi:C\to A$ satisfy 
$KL(\phi)=KL(\psi)$, $T(\phi)=T(\psi)$. 
Let $L\subset U_\infty(C)$ be a finite subset 
which generates $K_1(C)$. 
For any finite subset $F\subset C$ and $\ep>0$, 
there exists $\delta>0$ such that the following holds. 
If $\eta:K_1(C)\to\Aff(T(A))$ is a homomorphism satisfying 
\[
\eta(x)+\Ima D_A=\Theta_{\phi,\psi}(x)\quad \forall x\in K_1(C)
\]
and 
\[
\lVert\eta([w])\rVert<\delta\quad \forall w\in L, 
\]
then there exists a unitary $u\in A$ such that 
\[
\lVert\phi(f)-u\psi(f)u^*\rVert<\ep\quad \forall f\in F
\]
and 
\[
\frac{1}{2\pi\sqrt{-1}}\tau(\log(\phi(w)^*u\psi(w)u^*))
=\eta([w])(\tau)\quad \forall\tau\in T(A),\ w\in L. 
\]
\end{lem}
\begin{proof}
Applying Lemma \ref{adjust1} to $\psi$, $F$ and $\ep/2$, 
we obtain $\delta>0$. 
Suppose that $\eta\in\Hom(K_1(C),\Aff(T(A)))$ satisfies 
\[
\eta(x)+\Ima D_A=\Theta_{\phi,\psi}(x)\quad \forall x\in K_1(C)
\]
and 
\[
\lVert\eta([w])\rVert<\delta/2\quad \forall w\in L. 
\]
Choose a large finite subset $F_0\subset C$ and 
a small real number $\ep_0>0$. 
By virtue of Corollary \ref{rr0uniqueAH}, 
there exists a unitary $u_1\in A$ such that 
\[
\lVert\phi(f)-u_1\psi(f)u_1^*\rVert<\min\{\ep_0,\ep/2\}\quad 
\forall f\in F_0\cup F. 
\]
Put $\psi'=\Ad u_1\circ\psi$. 
For each $w\in U_\infty(C)$ satisfying $\lVert\phi(w)-\psi'(w)\rVert<2$, 
the function 
\[
z_w:\tau\mapsto\frac{1}{2\pi\sqrt{-1}}\tau(\log(\phi(w)^*\psi'(w)))
\]
gives an element of $\Aff(T(A))$. 
By \cite[Lemma 1]{HS}, we can see the following 
(see also the proof of Lemma \ref{Theta}). 
\begin{itemize}
\item If $w_1,w_2\in U_\infty(C)$ satisfy 
$\lVert\phi(w_1)-\psi'(w_1)\rVert+\lVert\phi(w_2)-\psi'(w_2)\rVert<2$, then 
$z_{w_1w_2}=z_{w_1}+z_{w_2}$. 
\item If $w:[0,1]\to U(C\otimes M_n)$ is a path of unitaries 
satisfying $\lVert\phi(w(t))-\psi'(w(t))\rVert<2$, then $z_{w(0)}=z_{w(1)}$. 
\end{itemize}
Therefore, if $F_0$ is large enough and $\ep_0$ is small enough, then 
there exists a homomorphism $\zeta:K_1(C)\to\Aff(T(A))$ such that 
\[
\zeta([w])(\tau)=z_w(\tau)
=\frac{1}{2\pi\sqrt{-1}}\tau(\log(\phi(w)^*\psi'(w)))\quad 
\forall\tau\in T(A),\ w\in L. 
\]
Clearly we may further assume that 
$\phi(w)$ and $u_1\psi(w)u_1^*$ are close enough 
to imply $\lVert\zeta([w])\rVert<\delta/2$ for every $w\in L$. 
We also have $\eta([w])-\zeta([w])\in\Ima D_A$ by Lemma \ref{Theta} (2). 
Hence there exists $\xi\in\Hom(K_1(C),K_0(A))$ such that 
$D_A(\xi(x))=\eta(x)-\zeta(x)$ for any $x\in K_1(C)$ 
and $\xi(x)=0$ for any $x\in\Tor(K_1(C))$. 
Moreover one has $\lVert D_A(\xi([w]))\rVert<\delta/2+\delta/2=\delta$. 
It follows from Lemma \ref{adjust1} that 
there exists a unitary $u_2\in A$ such that 
\[
\lVert[\psi(f),u_2]\rVert<\ep/2\quad \forall f\in F
\]
and 
\[
\Bott(\psi,u_2)(w)=\xi([w])\quad \forall w\in L. 
\]
Set $u=u_1u_2$. 
It is straightforward to check that 
\[
\lVert\phi(f)-u\psi(f)u^*\rVert<\ep
\]
holds for any $f\in F$. 
Besides, for any $\tau\in T(A)$ and $w\in L$, 
\begin{align*}
\tau(\log(\phi(w)^*u\psi(w)u^*))
&=\tau(\log(\phi(w)^*u_1u_2\psi(w)u_2^*u_1^*)) \\
&=\tau(\log(\phi(w)^*u_1\psi(w)u_1^*u_1\psi(w)^*u_2\psi(w)u_2^*u_1^*)) \\
&=\tau(\log(\phi(w)^*u_1\psi(w)u_1^*))
+\tau(\log(\psi(w)^*u_2\psi(w)u_2^*)) \\
&=2\pi\sqrt{-1}(\zeta([w])(\tau)+D_A(\Bott(\psi,u_2)(w))(\tau)) \\
&=2\pi\sqrt{-1}(\zeta([w])(\tau)+D_A(\xi([w]))(\tau)) \\
&=2\pi\sqrt{-1}\eta([w])(\tau), 
\end{align*}
where we have used \cite[Theorem 3.6]{L09AJM}. 
\end{proof}

The following lemma is an easy exercise and we leave it to the reader. 

\begin{lem}\label{decomp}
Let $L$ be a finitely generated abelian group and 
let $M$ be an abelian group. 
Let $N_0$ and $N_1$ be subgroups of $\Q$ and 
let $N\subset\Q$ be the subgroup generated by $N_0, N_1$. 
Then for any $\xi\in\Hom(L,M\otimes N)$, 
there exist $\xi_j\in\Hom(L,M\otimes N_j)$ such that $\xi=\xi_1-\xi_0$. 
\end{lem}

For each infinite supernatural number $\mathfrak{p}$ 
we let $M_{\mathfrak{p}}$ denote the UHF algebra of type $\mathfrak{p}$. 
Let $\mathfrak{p},\mathfrak{q}$ be 
relatively prime infinite supernatural numbers such that 
$M_{\mathfrak{p}}\otimes M_{\mathfrak{q}}\cong Q$. 
As in \cite{RW}, define a $C^*$-algebra $Z$ by 
\[
Z=\{f\in C([0,1],M_{\mathfrak{p}}\otimes M_{\mathfrak{q}})\mid
f(0)\in M_{\mathfrak{p}}\otimes\C,\ f(1)\in\C\otimes M_{\mathfrak{q}}\}. 
\]
The following proposition is 
the main part of the proof of Theorem \ref{main}. 

\begin{prop}\label{AtimesZ}
Let $X$ be a connected finite CW complex and 
let $A\in\mathcal{C}\cup\mathcal{C}'$. 
Suppose that two unital monomorphisms $\phi,\psi:C(X)\to A$ satisfy 
$KL(\phi)=KL(\psi)$, $T(\phi)=T(\psi)$ 
and $\Ima\Theta_{\phi,\psi}\subset\overline{\Ima D_A}$. 
Then for any finite subset $F\subset C(X)$ and $\ep>0$, 
there exists a unitary $u\in A\otimes Z$ such that 
\[
\lVert\phi(f)\otimes1-u(\psi(f)\otimes1)u^*\rVert<\ep
\]
holds for any $f\in F$. 
\end{prop}
\begin{proof}
We write $Q=M_{\mathfrak{p}}\otimes M_{\mathfrak{q}}$, 
$B_0=M_{\mathfrak{p}}\otimes\C$ and $B_1=\C\otimes M_{\mathfrak{q}}$. 
By Remark \ref{fourclasses}, 
$A\otimes Q$, $A\otimes B_0$ and $A\otimes B_1$ are 
in $\mathcal{T}\cup\mathcal{T}'$. 
Set $\bar\phi(f)=\phi(f)\otimes1$ and $\bar\psi(f)=\psi(f)\otimes1$. 
We regard $\bar\phi$ and $\bar\psi$ 
as homomorphisms from $C(X)$ into $A\otimes Q$ or $A\otimes B_j$. 
We identify $T(A\otimes Q)$, $T(A\otimes B_j)$ with $T(A)$. 
In the same way as Lemma \ref{adjust2}, to simplify notation, 
for $u\in A$ we denote $u\otimes1\in A\otimes M_n$ by $u$. 
Similarly, for $\tau\in T(A)$, 
$\tau\otimes\Tr$ on $A\otimes M_n$ is written by $\tau$ for short. 

Applying Theorem \ref{Basic} to 
$\bar\psi:C(X)\to A\otimes Q$, $F$ and $\ep/2$, 
we obtain a finite subset $L\subset\mathcal{K}(C(X))$, 
a finite subset $G_1\subset C(X)$ and $\delta_1>0$. 
By Remark \ref{Qstable}, we may and do assume that 
$L$ is written as $L=L_0\cup L_1$, 
where $L_0$ is a finite subset of $P(C(X)\otimes\K)$ and 
$L_1$ is a finite subset of $U_\infty(C(X))$. 
We may further assume that $L_1$ generates $K_1(C(X))$. 
Since $K_i(C(X))$ is finitely generated, 
one can find a finite subset $G_2\subset C(X)$ and $\delta_2>0$ 
such that the following holds: 
For any unitary $w\in A\otimes Q$ satisfying 
$\lVert[\bar\psi(f),w]\rVert<\delta_2$ for any $f\in G_2$, 
there exist $\xi_i\in\Hom(K_i(C(X)),K_{1-i}(A\otimes Q))$ such that 
$\xi_i([s])=\Bott(\bar\psi,w)(s)$ for any $s\in L_i$ and $i=0,1$ 
(\cite[Section 2]{L0612}). 
We may assume that $G_2$ contains $F\cup G_1$ and 
that $\delta_2$ is less than $\min\{\ep/2,\delta_1/2\}$. 
By applying Lemma \ref{adjust2} to 
$\bar\phi,\bar\psi:C(X)\to A\otimes B_j$, $G_2\subset C(X)$ and $\delta_2/2$, 
we get $\delta_{3,j}>0$ for each $j=0,1$. 

Since $K_1(C(X))$ is finitely generated and 
the homomorphism $\Theta_{\phi,\psi}$ factors through 
$K_1(C(X))/\Tor(K_1(C(X)))$ by Lemma \ref{Theta} (4), 
there exists $\eta\in\Hom(K_1(C(X)),\Aff(T(A)))$ such that 
\[
\eta(x)+\Ima D_A=\Theta_{\phi,\psi}(x)\quad \forall x\in K_1(C(X)). 
\]
Moreover, we may assume 
$\lVert\eta([w])\rVert<\min\{\delta_{3,0},\delta_{3,1}\}$ for all $w\in L_1$ 
because $\Ima\Theta_{\phi,\psi}$ is contained in the closure of $\Ima D_A$. 
It follows from Lemma \ref{adjust2} that 
there exists a unitary $u_j\in A\otimes B_j$ such that 
\[
\lVert\bar\phi(f)-u_j\bar\psi(f)u_j^*\rVert<\delta_2/2\quad \forall f\in G_2
\]
and 
\[
\frac{1}{2\pi\sqrt{-1}}\tau(\log(\bar\phi(w)^*u_j\bar\psi(w)u_j^*))
=\eta([w])(\tau)\quad \forall\tau\in T(A\otimes B_j),\ w\in L_1. 
\]
In particular 
one has $\lVert[\bar\psi(f),u_1^*u_0]\rVert<\delta_2$ for $f\in G_2$. 
From the choice of $G_2$ and $\delta_2$, 
we can find $\xi_i\in\Hom(K_i(C(X)),K_{1-i}(A\otimes Q))$ such that 
$\xi_i([s])=\Bott(\bar\psi,u_1^*u_0)(s)$ holds for any $s\in L_i$ and $i=0,1$. 
By \cite[Theorem 3.6]{L09AJM}, 
\begin{align*}
D_{A\otimes Q}(\xi_1([w]))(\tau)
&=D_{A\otimes Q}(\Bott(\bar\psi,u_1^*u_0)(w))(\tau) \\
&=\frac{1}{2\pi\sqrt{-1}}
\tau(\log(u_1^*u_0\bar\psi(w)u_0^*u_1\bar\psi(w)^*)) \\
&=\frac{1}{2\pi\sqrt{-1}}\left(
\tau(\log(\bar\phi(w)^*u_0\bar\psi(w)u_0^*))
-\tau(\log(\bar\phi(w)^*u_1\bar\psi(w)u_1^*))\right) \\
&=\eta([w])(\tau)-\eta([w])(\tau)=0
\end{align*}
for any $\tau\in T(A\otimes Q)$. 
Thus $\Ima\xi_1$ is contained in $\Ker D_{A\otimes Q}$. 
By Lemma \ref{decomp}, 
we can find $\xi_{1,j}:K_1(C(X))\to\Ker D_{A\otimes B_j}$ 
such that $\xi_1=\xi_{1,1}-\xi_{1,0}$, 
where $\Ker D_{A\otimes C}$ is naturally identified 
with $(\Ker D_A)\otimes K_0(C)$ for $C=Q,B_0,B_1$. 
In the same way, 
one obtains $\xi_{0,j}:K_0(C(X))\to K_1(A\otimes B_j)$ 
such that $\xi_0=\xi_{0,1}-\xi_{0,0}$. 

We consider the following exact sequence of $C^*$-algebras: 
\[
\begin{CD}0@>>>C_0(X\times(\T\setminus\{-1\}))
@>\iota>>C(X\times\T)@>\pi>>C(X)@>>>0, \end{CD}
\]
where $\pi$ is the evaluation at $-1\in\T$. 
We write $S=C_0(\T\setminus\{-1\})$ for short. 
Let $\rho:C(X)\to C(X\times\T)$ be the homomorphism 
defined by $\rho(f)=f\otimes1$. 
Then $\pi\circ\rho$ is the identity on $C(X)$. 
This split exact sequence induces the isomorphism 
\[
(a,b)\mapsto KL(\rho)(a)+KL(\iota)(b)
\]
from $\underline{K}(C(X))\oplus\underline{K}(C(X)\otimes S)$ 
to $\underline{K}(C(X\times\T))$. 
Let $\omega_i:K_i(C(X)\otimes S)\to K_{1-i}(C(X))$ 
be the canonical isomorphism for each $i=0,1$. 
For each $j=0,1$, choose $\kappa_j\in KL(C(X)\otimes S,A\otimes B_j)$ 
such that $K_i(\kappa_j)=\xi_{1-i,j}\circ\omega_i$. 
Define $\tilde\kappa_j\in KL(C(X\times\T),A\otimes B_j)$ by 
\[
\tilde\kappa_j\circ KL(\rho)=KL(\bar\psi)\quad\text{and}\quad
\tilde\kappa_j\circ KL(\iota)=\kappa_j. 
\]
Clearly $K_0(\tilde\kappa_j)$ is unital. 
Also, 
$K_0(\tilde\kappa_j)\circ K_0(\rho)=K_0(\bar\psi)$ is (strictly) positive 
and the image of 
\[
K_0(\tilde\kappa_j)\circ K_0(\iota)=K_0(\kappa_j)=\xi_{1,j}\circ\omega_0
\]
is contained in $\Ker D_{A\otimes B_j}$. 
It follows from Lemma \ref{KL+1} that 
$K_0(\tilde\kappa_j)$ is unital and (strictly) positive. 
Thus, $\tilde\kappa_j$ is in $KL(C(X\times\T),A\otimes B_j)_{+,1}$. 
Let $\tau_0\in T(C(\T))$ be the tracial state 
corresponding to the Haar measure on $\T$ and 
define the affine continuous map 
$\lambda:T(A\otimes B_j)\to T(C(X\times\T))$ 
by $\lambda(\tau)=T(\bar\psi)(\tau)\otimes\tau_0$. 
Thanks to Theorem \ref{NWexist}, 
there exists a unital monomorphism $\sigma_j:C(X\times\T)\to A\otimes B_j$ 
such that $KL(\sigma_j)=\tilde\kappa_j$ and $T(\sigma_j)=\lambda$. 
Since $KL(\sigma_j\circ\rho)=KL(\bar\psi)$ and 
$T(\sigma_j\circ\rho)=T(\bar\psi)$, 
$\bar\psi$ and $\sigma_j\circ\rho$ are approximately unitarily equivalent 
by Theorem \ref{rr0unique}. 
Hence there exists a unitary $v_j\in A\otimes B_j$ such that 
\[
\lVert[\bar\psi(f),v_j]\rVert<\delta_2/2\quad 
\forall f\in G_2
\]
and 
\begin{align*}
\Bott(\bar\psi,v_j)(s)
&=(K_{1-i}(\sigma_j)\circ K_{1-i}(\iota)\circ\omega_{1-i}^{-1})([s]) \\
&=(K_{1-i}(\kappa_j)\circ\omega_{1-i}^{-1})([s]) \\
&=(\xi_{i,j}\circ\omega_{1-i}\circ\omega_{1-i}^{-1})([s]) \\
&=\xi_{i,j}([s])
\end{align*}
for any $s\in L_i$ and $i=0,1$. 

It is easy to see that 
\[
\lVert\bar\phi(f)-u_jv_j\bar\psi(f)v_j^*u_j^*\rVert
<\delta_2/2+\delta_2/2=\delta_2
\]
holds for any $f\in G_2$. 
In particular one has 
\[
\lVert[\bar\psi(f),v_1^*u_1^*u_0v_0]\rVert<2\delta_2<\delta_1\quad 
\forall f\in G_2. 
\]
Besides, 
when $G_2$ is sufficiently large and $\delta_2$ is sufficiently small, 
we get 
\begin{align*}
\Bott(\bar\psi,v_1^*u_1^*u_0v_0)([s])
&=\Bott(\bar\psi,v_1^*)([s])+\Bott(\bar\psi,u_1^*u_0)([s])
+\Bott(\bar\psi,v_0)([s]) \\
&=-\xi_{i,1}([s])+\xi_i([s])+\xi_{i,0}([s])=0
\end{align*}
for any $s\in L_i$ and $i=0,1$, 
where we have used \cite[(e2.6)]{L09AJM}. 
Therefore, by Theorem \ref{Basic}, 
we can find a path of unitaries $w:[0,1]\to A\otimes Q$ such that 
$w(0)=v_1^*u_1^*u_0v_0$, $w(1)=1$ and 
\[
\lVert[\bar\psi(f),w(t)]\rVert<\ep/2\quad \forall f\in F,\ t\in[0,1]. 
\]
Define a unitary $U\in Z$ by $U(t)=u_1v_1w(t)$. 
It is easy to see that 
\[
\lVert\phi(f)\otimes1-U(\psi(f)\otimes1)U^*\rVert<\ep/2+\delta_2<\ep
\]
holds for any $f\in F$. 
\end{proof}

\begin{thm}\label{main}
Let $X$ be a compact metrizable space and 
let $A\in\mathcal{C}\cup\mathcal{C}'$. 
For unital monomorphisms $\phi,\psi:C(X)\to A$, 
the following two conditions are equivalent. 
\begin{enumerate}
\item $\phi$ and $\psi$ are approximately unitarily equivalent. 
\item $KL(\phi)=KL(\psi)$, 
$\tau\circ\phi=\tau\circ\psi$ for any $\tau\in T(A)$ 
and $\Ima\Theta_{\phi,\psi}\subset\overline{\Ima D_A}$. 
\end{enumerate}
\end{thm}
\begin{proof}
It is straightforward to check that (1) implies (2). 
Indeed, $KL(\phi)=KL(\psi)$ and $T(\phi)=T(\psi)$ are clear. 
By Lemma \ref{Theta}, 
$\Ima\Theta_{\phi,\psi}\subset\overline{\Ima D_A}$ also follows. 

We would like to show the other implication. 
We first consider the case that $X$ is a connected finite CW complex. 
Let $\phi,\psi:C(X)\to A$ be unital monomorphisms 
satisfying $KL(\phi)=KL(\psi)$, $T(\phi)=T(\psi)$ and 
$\Ima\Theta_{\phi,\psi}\subset\overline{\Ima D_A}$. 
We may replace the target algebra $A$ with $A\otimes\mathcal{Z}$, 
because $A$ is $\mathcal{Z}$-absorbing. 
Since $\mathcal{Z}$ is strongly self-absorbing (\cite{TW07TAMS}), 
there exists an approximately inner endomorphism 
$\pi:A\otimes\mathcal{Z}\to A\otimes\mathcal{Z}$ 
such that $\pi(A\otimes\mathcal{Z})=A\otimes\C$. 
Hence $\phi$ and $\pi\circ\phi$ (resp. $\psi$ and $\pi\circ\psi$) are 
approximately unitarily equivalent. 
By \cite[Proposition 3.3]{RW}, 
the $C^*$-algebra $Z$ embeds unitally into $\mathcal{Z}$. 
It follows from Proposition \ref{AtimesZ} that 
$\pi\circ\phi$ and $\pi\circ\psi$ are 
approximately unitarily equivalent. 
Therefore $\phi$ and $\psi$ are approximately unitarily equivalent. 

A general finite CW complex is 
a finite union of pairwise disjoint connected finite CW complexes. 
Since $A$ has cancellation by \cite[Theorem 6.7]{R04IJM} 
and $eAe$ is in $\mathcal{C}\cup\mathcal{C}'$ 
for any non-zero projection $e\in A$, 
the conclusion follows from the previous case. 

Let $X$ be a compact metrizable space. 
Let $\{f_1,f_2,\dots,f_n\}$ be a finite subset of $C(X)$ and let $\ep>0$. 
By \cite[Lemma 1]{Mardesic}, 
there exist a finite CW complex (actually a finite simplicial complex) $Y$, 
a finite subset $\{g_1,g_2,\dots,g_n\}$ of $C(Y)$ 
and a unital monomorphism $\sigma:C(Y)\to C(X)$ such that 
$\lVert f_i-\sigma(g_i)\rVert<\ep/3$ for any $i=1,2,\dots,n$. 
Clearly $KL(\phi\circ\sigma)=KL(\psi\circ\sigma)$, 
$T(\phi\circ\sigma)=T(\psi\circ\sigma)$ and 
$\Ima\Theta_{\phi\circ\sigma,\psi\circ\sigma}$ is contained 
in the closure of $\Ima D_A$. 
It follows from the argument above that 
$\phi\circ\sigma$ and $\psi\circ\sigma$ are 
approximately unitarily equivalent. 
Hence there exists a unitary $u\in A$ such that 
$\lVert\phi(\sigma(g_i))-u\psi(\sigma(g_i))u^*\rVert<\ep/3$, 
which implies $\lVert\phi(f_i)-u\psi(f_i)u^*\rVert<\ep$. 
Thus, $\phi$ and $\psi$ are approximately unitarily equivalent. 
\end{proof}

\begin{rem}
In the theorem above, if $A$ has real rank zero, then 
the image of $D_A$ is dense in $\Aff(T(A))$. 
Hence the condition $\Ima\Theta_{\phi,\psi}\subset\overline{\Ima D_A}$ 
is trivially satisfied. 
\end{rem}

\begin{cor}\label{mainAH}
Let $C$ be a unital AH algebra and 
let $A\in\mathcal{C}\cup\mathcal{C}'$. 
For unital monomorphisms $\phi,\psi:C\to A$, 
the following two conditions are equivalent. 
\begin{enumerate}
\item $\phi$ and $\psi$ are approximately unitarily equivalent. 
\item $KL(\phi)=KL(\psi)$, 
$\tau\circ\phi=\tau\circ\psi$ for any $\tau\in T(A)$ 
and $\Ima\Theta_{\phi,\psi}\subset\overline{\Ima D_A}$. 
\end{enumerate}
\end{cor}
\begin{proof}
We can prove this in the same way as Corollary \ref{rr0uniqueAH}. 
\end{proof}

%%%%%%%%%%%%%%%%%%%%%%%%%%%%%%%%%%%%%%%%%%%%%%%%%%%%%%%%%%%%
\section{Homomorphisms between simple $\mathcal{Z}$-stable $C^*$-algebras}

In this section we prove Theorem \ref{main2}. 
The main idea is almost the same as 
Proposition \ref{AtimesZ} and Theorem \ref{main}. 
The proof is, however, somewhat lengthy 
because we must work with finitely generated subgroups of $K_*(C)$ 
so as to use Lemma \ref{decomp}. 

\begin{thm}\label{main2}
Let $C$ be a nuclear $C^*$-algebra in $\mathcal{C}$ satisfying the UCT and 
let $A\in\mathcal{C}\cup\mathcal{C}'$. 
For any unital homomorphisms $\phi,\psi:C\to A$, 
the following are equivalent. 
\begin{enumerate}
\item $\phi$ and $\psi$ are approximately unitarily equivalent. 
\item $KL(\phi)=KL(\psi)$ and 
$\Ima\Theta_{\phi,\psi}\subset\overline{\Ima D_A}$. 
\end{enumerate}
\end{thm}
\begin{proof}
The implication (1)$\Rightarrow$(2) is trivial. 
We would like to show the other implication. 
Note that $KL(\phi)=KL(\psi)$ implies $T(\phi)=T(\psi)$, 
because projections of $C$ separate traces. 
Since $\mathcal{Z}$ is strongly self-absorbing (\cite{TW07TAMS}) and 
$C$ (resp. $A$) is $\mathcal{Z}$-stable by assumption, 
there exists an isomorphisms $\pi_C:C\to C\otimes\mathcal{Z}$ 
(resp. $\pi_A:A\to A\otimes\mathcal{Z}$) 
which is approximately unitarily equivalent to 
the unital monomorphism $c\mapsto c\otimes1$. 
It is not so hard to see that any unital homomorphism $\phi:C\to A$ is 
approximately unitarily equivalent to 
$\pi_A^{-1}\circ(\phi\otimes\id)\circ\pi_C$. 
Hence, for given homomorphisms $\phi,\psi$ 
satisfying $KL(\phi)=KL(\psi)$ and 
$\Ima\Theta_{\phi,\psi}\subset\overline{\Ima D_A}$, 
it suffices to show that 
$\phi\otimes\id:C\otimes\mathcal{Z}\to A\otimes\mathcal{Z}$ is 
approximately unitarily equivalent to 
$\psi\otimes\id:C\otimes\mathcal{Z}\to A\otimes\mathcal{Z}$. 

Suppose that 
we are given a finite subset $F\subset C\otimes\mathcal{Z}$ and $\ep>0$. 
We would like to show $\phi\otimes\id\sim_{F,\ep}\psi\otimes\id$. 
Without loss of generality, 
we may assume that $F$ is contained in $C\otimes\C$. 

Let $Z,B_0,B_1,Q$ be as in Proposition \ref{AtimesZ}. 
We identify $T(A\otimes Q)$, $T(A\otimes B_j)$ with $T(A)$. 
Put $\bar\phi=\phi\otimes\id$, $\bar\psi=\psi\otimes\id$. 
As in the proof of Proposition \ref{AtimesZ}, 
we omit `$\otimes\C$', `$\otimes1$' and `$\otimes\Tr$' to simplify notation. 

By the classification theorem in \cite{L04Duke}, 
$C\otimes Q$ is a unital simple AT algebra with real rank zero. 
Thus, $C\otimes Q$ can be written as an inductive limit of $C^*$-algebras 
of the form $\bigoplus_{i=1}^nM_{k_i}(C(\T))$. 
By using Theorem \ref{Basic2} to 
$\bar\psi:C\otimes Q\to A\otimes Q$, $F$ and $\ep/2$, 
we obtain a finite subset $L\subset\mathcal{K}(C\otimes Q)$, 
a finite subset $G_1\subset C\otimes Q$ and $\delta_1>0$. 
By Remark \ref{Qstable}, we may and do assume that 
$L$ is written as $L=L_0\cup L_1$, 
where $L_0$ is a finite subset of $P(C\otimes Q\otimes\K)$ and 
$L_1$ is a finite subset of $U_\infty(C\otimes Q)$. 
We may further assume that $L_0$ and $L_1$ are 
finite subsets of $P(C\otimes\K)$ and $U_\infty(C)$ respectively, 
because $\Bott$ gives rise to a `partial homomorphism' (see Section 2.3). 
Let $H_i\subset K_i(C)$ be the subgroup generated by $L_i$. 
Since $H_i$ is finitely generated, 
one can find a finite subset $G_2\subset C\otimes Q$ and $\delta_2>0$ 
such that the following holds: 
For any unitary $w\in A\otimes Q$ satisfying 
$\lVert[\bar\psi(c),w]\rVert<\delta_2$ for any $c\in G_2$, 
there exist $\xi_i\in\Hom(H_i,K_{1-i}(A\otimes Q))$ such that 
$\xi_i([s])=\Bott(\bar\psi,w)(s)$ for any $s\in L_i$ and $i=0,1$ 
(\cite[Section 2]{L0612}). 
We may assume that $G_2$ contains $G_1$ and 
that $\delta_2$ is less than $\delta_1$. 
As $C\otimes Q$ is generated by $C$, $B_0$ and $B_1$, 
one may choose finite subsets 
$G_3\subset C$, $G_{3,0}\subset B_0$, $G_{3,1}\subset B_1$ and $\delta_3>0$ 
so that if a unitary $w\in A\otimes Q$ satisfies 
$\lVert[\bar\psi(c),w]\rVert<\delta_3$ 
for every $c\in G_3\cup G_{3,0}\cup G_{3,1}$, 
then $\lVert[\bar\psi(c),w]\rVert<\delta_2$ holds for all $c\in G_2$. 
We assume that $G_3$ contains $F$ and that $\delta_3$ is less than $\ep$. 
For each $j=0,1$, by the classification theorem in \cite{L04Duke}, 
$C\otimes B_j$ is a unital simple AH algebra 
with real rank zero and slow dimension growth, and so 
one can find a unital subalgebra $C_j\subset C\otimes B_j$ 
such that the following hold. 
\begin{itemize}
\item $C_j$ is a finite direct sum of $C^*$-algebras of 
the form $p(C(X)\otimes M_k)p$, 
where $X$ is a connected finite CW complex 
(with dimension at most three) and $p\in C(X)\otimes M_k$ is a projection. 
\item There exists a finite subset $G'_{3,j}\subset C_j$ such that 
any elements of $G_3\cup G_{3,j}$ are 
within distance $\delta_3/12$ of $G'_{3,j}$. 
\item There exist finite subsets 
$L'_{0,j}\in P(C_j\otimes\K)$ and $L'_{1,j}\subset U_\infty(C_j)$ such that 
any elements of $L_i$ are within distance $1/2$ of $L'_{i,j}$ 
for each $i=0,1$. 
Let $L_i\ni s\mapsto s'_j\in L'_{i,j}$ be a map 
such that $\lVert s-s'_j\rVert<1/2$. 
We further require that $L'_{1,j}$ generates $K_1(C_j)$. 
\end{itemize}
Let $\gamma_j:C_j\to C\otimes B_j$ denote the embedding map. 
For each $i=0,1$, 
we choose a finitely generated subgroup $H'_i\subset K_i(C)$ 
so that $H_i$ is contained in $H'_i$ and 
$\Ima K_i(\gamma_j)$ is contained in $H'_i\otimes K_i(B_j)$ for each $j=0,1$. 
Applying Lemma \ref{adjust2} to 
$\bar\phi\circ\gamma_j:C_j\to A\otimes B_j$, 
$\bar\psi\circ\gamma_j:C_j\to A\otimes B_j$, 
$G'_{3,j}$ and $\delta_3/12$, we get $\delta_{4,j}>0$ for each $j=0,1$. 

As in the proof of Proposition \ref{AtimesZ}, 
we can find a homomorphism $\eta:H'_1\to\Aff(T(A))$ such that 
\[
\eta(x)+\Ima D_A=\Theta_{\phi,\psi}(x)\quad \forall x\in H'_1
\]
and 
\[
\lVert\tilde\eta_j([\gamma_j(w)])\rVert<\delta_{4,j}\quad 
\forall w\in L'_{1,j},\ j=0,1, 
\]
where $\tilde\eta_j\in\Hom(H'_1\otimes K_0(B_j),\Aff(T(A)))$ denotes 
the homomorphism induced from $\eta$. 
It follows from Lemma \ref{adjust2} that 
there exists a unitary $u_j\in A\otimes B_j$ such that 
\[
\lVert\bar\phi(c)-u_j\bar\psi(c)u_j^*\rVert<\delta_3/12\quad 
\forall c\in G'_{3,j}
\]
and 
\[
\frac{1}{2\pi\sqrt{-1}}\tau(\log(\bar\phi(w)^*u_j\bar\psi(w)u_j^*))
=\tilde\eta_j([\gamma_j(w)])(\tau)\quad 
\forall\tau\in T(A\otimes B_j),\ w\in L'_{1,j}. 
\]
By choosing $G'_{3,j}$ large enough in advance, 
we may also assume that 
$\lVert\bar\phi(w)-u_j\bar\psi(w)u_j^*\rVert$ is less than $1/2$ 
for every $w\in L'_{1,j}$. 
From the choice of $G'_{3,j}$, 
we obtain $\lVert\bar\phi(c)-u_j\bar\psi(c)u_j^*\rVert<\delta_3/4$ 
for all $c\in G_3\cup G_{3,j}$. 
If $c$ is in $G_{3,1-j}$, then 
$u_j\in A\otimes B_j$ commutes with $\bar\psi(c)\in B_{1-j}$, 
and so $\bar\phi(c)=\bar\psi(c)=u_j\bar\psi(c)u_j^*$. 
Therefore 
\[
\lVert[\bar\psi(c),u_1^*u_0]\rVert<\delta_3/2\quad 
\forall c\in G_3\cup G_{3,0}\cup G_{3,1}. 
\]
From the choice of $G_3$, $G_{3,0}$, $G_{3,1}$ and $\delta_3$, one has 
\[
\lVert[\bar\psi(c),u_1^*u_0]\rVert<\delta_2\quad 
\forall c\in G_2. 
\]
Then, from the choice of $G_2$ and $\delta_2$, 
we can find $\xi_i\in\Hom(H_i,K_{1-i}(A\otimes Q))$ such that 
$\xi_i([s])=\Bott(\bar\psi,u_1^*u_0)(s)$ holds for any $s\in L_i$ and $i=0,1$. 
By \cite[Theorem 3.6]{L09AJM}, 
\begin{align*}
D_{A\otimes Q}(\xi_1([w]))(\tau)
&=D_{A\otimes Q}(\Bott(\bar\psi,u_1^*u_0)(w))(\tau) \\
&=\frac{1}{2\pi\sqrt{-1}}
\tau(\log(u_1^*u_0\bar\psi(w)u_0^*u_1\bar\psi(w)^*)) \\
&=\frac{1}{2\pi\sqrt{-1}}\left(
\tau(\log(\bar\phi(w)^*u_0\bar\psi(w)u_0^*))
-\tau(\log(\bar\phi(w)^*u_1\bar\psi(w)u_1^*))\right) \\
&=\frac{1}{2\pi\sqrt{-1}}\left(
\tau(\log(\bar\phi(w'_0)^*u_0\bar\psi(w'_0)u_0^*))
-\tau(\log(\bar\phi(w'_1)^*u_1\bar\psi(w'_1)u_1^*))\right) \\
&=\tilde\eta_0([w'_0])(\tau)-\tilde\eta_1([w'_1])(\tau) \\
&=\tilde\eta_0([w])(\tau)-\tilde\eta_1([w])(\tau) \\
&=\eta([w])(\tau)-\eta([w])(\tau)=0
\end{align*}
for any $w\in L_1$ and $\tau\in T(A\otimes Q)$. 
Thus $\Ima\xi_1$ is contained in $\Ker D_{A\otimes Q}$. 
Since $\Ker D_{A\otimes Q}=(\Ker D_A)\otimes\Q$ is divisible, 
$\xi_1$ extends to a homomorphism $\xi_1:H'_1\to\Ker D_{A\otimes Q}$. 
Likewise $\xi_0$ extends to a homomorphism 
$\xi_0:H'_0\to K_1(A\otimes Q)=K_1(A)\otimes\Q$. 
By Lemma \ref{decomp}, 
we can find $\xi_{1,j}:H'_1\to\Ker D_{A\otimes B_j}$ 
such that $\xi_1=\xi_{1,1}-\xi_{1,0}$. 
We let $\tilde\xi_{1,j}:H'_1\otimes K_0(B_j)\to\Ker D_{A\otimes B_j}$ 
denote the homomorphism induced from $\xi_{1,j}$. 
In the same way, 
one obtains $\xi_{0,j}:H'_0\to K_1(A\otimes B_j)$ 
such that $\xi_0=\xi_{0,1}-\xi_{0,0}$. 
We let $\tilde\xi_{0,j}:H'_0\otimes K_0(B_j)\to K_1(A\otimes B_j)$ 
denote the homomorphism induced from $\xi_{0,j}$. 

In the same way as in the proof of Proposition \ref{AtimesZ}, 
for each $j=0,1$, 
we consider the following exact sequence of $C^*$-algebras: 
\[
\begin{CD}0@>>>C_j\otimes C_0(\T\setminus\{-1\})
@>\iota_j>>C_j\otimes C(\T)@>\pi_j>>C_j@>>>0, \end{CD}
\]
where $\pi_j$ is the evaluation at $-1\in\T$. 
We write $S=C_0(\T\setminus\{-1\})$ for short. 
Let $\rho_j:C_j\to C_j\otimes C(\T)$ be the homomorphism 
defined by $\rho_j(c)=c\otimes 1$. 
Then $\pi_j\circ\rho_j$ is the identity on $C_j$. 
This split exact sequence induces the isomorphism 
\[
(a,b)\mapsto KL(\rho_j)(a)+KL(\iota_j)(b)
\]
from $\underline{K}(C_j)\oplus\underline{K}(C_j\otimes S)$ 
to $\underline{K}(C_j\otimes C(\T))$. 
Let $\omega_{i,j}:K_i(C_j\otimes S)\to K_{1-i}(C_j)$ 
be the canonical isomorphism for each $i,j=0,1$. 
For each $j=0,1$, choose $\kappa_j\in KL(C_j\otimes S,A\otimes B_j)$ 
so that 
\[
K_i(\kappa_j)=\tilde\xi_{1-i,j}\circ K_{1-i}(\gamma_j)\circ\omega_{i,j}\quad 
\forall i=0,1. 
\]
Notice that the composition of $\tilde\xi_{1-i,j}$ and $K_{1-i}(\gamma_j)$ 
is well-defined, 
because $\Ima K_{1-i}(\gamma_j)$ is contained in $H'_{1-i}\otimes K_0(B_j)$. 
Define $\tilde\kappa_j\in KL(C_j\otimes C(\T),A\otimes B_j)$ by 
\[
\tilde\kappa_j\circ KL(\rho_j)
=KL(\bar\psi\circ\gamma_j)\quad\text{and}\quad
\tilde\kappa_j\circ KL(\iota_j)=\kappa_j. 
\]
Clearly $K_0(\tilde\kappa_j)$ is unital. 
Also for any $x\in K_0(C_j\otimes C(\T))_+\setminus\{0\}$, 
one has $K_0(\pi_j)(x)\in K_0(C_j)_+\setminus\{0\}$, and so 
$\tau(K_0(\bar\psi\circ\gamma_j\circ\pi_j)(x))>0$ 
for every $\tau\in T(A\otimes B_j)$. 
Since the image of $\tilde\xi_{1,j}$ is contained 
in the kernel of $D_{A\otimes B_j}$, 
we obtain 
\[
\tau(K_0(\tilde\kappa_j)(x))=\tau(K_0(\bar\psi\circ\gamma_j\circ\pi_j)(x))>0, 
\]
which entails $K_0(\tilde\kappa_j)(x)\in K_0(A\otimes B_j)_+\setminus\{0\}$. 
It thus follows that 
$K_0(\tilde\kappa_j)$ is unital and strictly positive, 
and hence $\tilde\kappa_j$ is in $KL(C_j\otimes C(\T),A\otimes B_j)_{+,1}$. 
Let $\tau_0\in T(C(\T))$ be the tracial state 
corresponding to the Haar measure on $\T$ and 
define the affine continuous map 
$\lambda_j:T(A\otimes B_j)\to T(C_j\otimes C(\T))$ 
by $\lambda_j(\tau)=T(\bar\psi\circ\gamma_j)(\tau)\otimes\tau_0$. 
For each minimal central projection $p\otimes1\in C_j\otimes C(\T)$ 
and $\tau\in T(A\otimes B_j)$, 
it is easy to verify 
\[
\tau(K_0(\tilde\kappa_j)([p\otimes1]))
=\tau(K_0(\bar\psi\circ\gamma_j)([\pi_j(p\otimes1)]))
=\tau(\bar\psi(\gamma_j(p)))
=\lambda_j(\tau)(p\otimes1). 
\]
Hence the hypotheses of Corollary \ref{NWexist2} are satisfied. 
Thanks to Corollary \ref{NWexist2}, 
there exists a unital monomorphism $\sigma_j:C_j\otimes C(\T)\to A\otimes B_j$ 
such that $KL(\sigma_j)=\tilde\kappa_j$ and $T(\sigma_j)=\lambda_j$. 
Since $KL(\sigma_j\circ\rho_j)=KL(\bar\psi\circ\gamma_j)$ and 
$T(\sigma_j\circ\rho_j)=T(\bar\psi\circ\gamma_j)$, 
Corollary \ref{rr0uniqueAH} implies that 
$\bar\psi\circ\gamma_j$ and $\sigma_j\circ\rho_j$ are 
approximately unitarily equivalent. 
Hence there exists a unitary $v_j\in A\otimes B_j$ such that 
\[
\lVert[\bar\psi(\gamma_j(c)),v_j]\rVert<\delta_3/12\quad 
\forall c\in G'_{3,j}
\]
and 
\begin{align*}
\Bott(\bar\psi\circ\gamma_j,v_j)(s)
&=(K_{1-i}(\sigma_j)\circ K_{1-i}(\iota_j)\circ\omega_{1-i,j}^{-1})([s]) \\
&=(K_{1-i}(\kappa_j)\circ\omega_{1-i,j}^{-1})([s]) \\
&=(\tilde\xi_{i,j}\circ K_i(\gamma_j)\circ\omega_{1-i,j}
\circ\omega_{1-i,j}^{-1})([s]) \\
&=\tilde\xi_{i,j}([\gamma_j(s)])
\end{align*}
for any $s\in L'_{i,j}$ and $i=0,1$. 
As before, 
\[
\lVert[\bar\psi(c),v_j]\rVert<\delta_3/4
\]
holds for any $c\in G_3\cup G_{3,0}\cup G_{3,1}$. 
By choosing $G'_{3,j}$ large enough and $\delta_3$ small enough in advance, 
we have $\Bott(\bar\psi,v_j)(s)=\Bott(\bar\psi\circ\gamma_j,v_j)(s'_j)$ 
for any $s\in L_i$ and $i=0,1$. 

It is easy to see that 
\[
\lVert\bar\phi(c)-u_jv_j\bar\psi(c)v_j^*u_j^*\rVert
<\delta_3/4+\delta_3/4=\delta_3/2
\]
holds for any $c\in G_3\cup G_{3,0}\cup G_{3,1}$. 
In particular one has 
\[
\lVert[\bar\psi(c),v_1^*u_1^*u_0v_0]\rVert<\delta_3\quad 
\forall c\in G_3\cup G_{3,0}\cup G_{3,1}, 
\]
and hence 
\[
\lVert[\bar\psi(c),v_1^*u_1^*u_0v_0]\rVert<\delta_2<\delta_1\quad 
\forall c\in G_1, 
\]
because $G_1$ is contained in $G_2$. 
Besides, 
when $G_3$ is sufficiently large and $\delta_3$ is sufficiently small, 
we get 
\begin{align*}
\Bott(\bar\psi,v_1^*u_1^*u_0v_0)([s])
&=\Bott(\bar\psi,v_1^*)([s])+\Bott(\bar\psi,u_1^*u_0)([s])
+\Bott(\bar\psi,v_0)([s]) \\
&=-\tilde\xi_{i,1}([s'_1])+\xi_i([s])+\tilde\xi_{i,0}([s'_0]) \\
&=-\xi_{i,1}([s])+\xi_i([s])+\xi_{i,0}([s])=0
\end{align*}
for any $s\in L_i$ and $i=0,1$, 
where we have used \cite[(e2.6)]{L09AJM}. 
Therefore, by Theorem \ref{Basic2}, 
we can find a path of unitaries $w:[0,1]\to A\otimes Q$ such that 
$w(0)=v_1^*u_1^*u_0v_0$, $w(1)=1$ and 
\[
\lVert[\bar\psi(c),w(t)]\rVert<\ep/2\quad \forall c\in F,\ t\in[0,1]. 
\]
Define a unitary $U\in Z$ by $U(t)=u_1v_1w(t)$. 
It is easy to see that 
\[
\lVert\phi(c)\otimes1-U(\psi(c)\otimes1)U^*\rVert<\ep/2+\delta_3/2<\ep
\]
holds for any $c\in F$. 
\end{proof}

\begin{rem}
In the theorem above, if $A$ has real rank zero, then 
the image of $D_A$ is dense in $\Aff(T(A))$. 
Hence the condition $\Ima\Theta_{\phi,\psi}\subset\overline{\Ima D_A}$ 
is trivially satisfied. 
\end{rem}

\end{document}